\documentclass[a4paper,10pt, notitlepage]{article}
\usepackage[T1]{fontenc}
\usepackage[cp1250]{inputenc}
\usepackage[english]{babel}
\usepackage{amsfonts}
\usepackage{amsmath}
\usepackage{color}
\usepackage{enumerate}
\usepackage{graphicx}
\usepackage{multicol}
\usepackage{longtable}
\usepackage{caption}
\usepackage{subcaption}

\newtheorem{definition}{Definition}[section]
\newtheorem{theorem}{Theorem}[section]

\newtheorem{coro}[theorem]{Corollary}

\newenvironment{proof}{\textsc{Proof.} }{\begin{flushright}$\Box$\end{flushright}}
\newenvironment{remark}{\emph{Remark.} }{\par}

\newcommand{\hnorm}[1]{\left|\, #1\, \right|}

\newcommand{\xnorm}[2]{\left\|\, #1\, \right\|_{#2}}

\newcommand{\dual}[1]{{#1}^*}

\newcommand{\Ha}[1]{H^{#1}(\Omega)}

\newcommand{\Lp}[2]{L^{#1}(#2)}
\newcommand{\Ce}[1]{C([0,T];#1)}
\newcommand{\eLp}[2]{L^{#1}(0,T;#2)}

\newcommand{\Cont}[1]{C(#1)}

\newcommand{\Lipfunc}[1]{Lip(#1)}

\newcommand{\R}{\mathbb{R}}

\newcommand{\asmstylei}[1]{(\Roman{#1})}
\newcommand{\asmstyleii}[1]{\alph{#1})}

\newcommand{\asmlabeli}[1]
{\newcounter{#1}\setcounter{#1}{\value{enumi}}}

\newenvironment{asmlisti}{
\begin{list}{\asmstylei{enumi}}{\usecounter{enumi}}
\setcounter{enumi}{\value{tempenumi}}
}
{
\setcounter{tempenumi}{\value{enumi}}
\end{list}
}

{
\end{list}
}

{
\end{list}
}

\newcommand{\tcr}[1]{{#1}}
\newcommand{\tcb}[1]{{#1}}

\begin{document}
\renewcommand{\theequation}{\thesection.\arabic{equation}}
\newlength{\asmlength}
\newcounter{tempenumi}
\setcounter{tempenumi}{0}

\providecommand{\sep}{, }

\newcommand{\fmauthor}{Grzegorz Dudziuk}
\newcommand{\fminstitution}{Institute of Applied Mathematics and Mechanics, University of Warsaw}
\newcommand{\fmaddress}{Banacha 2, 02-097 Warsaw, Poland}
\newcommand{\fmemail}{grzegorz.dudziuk@mimuw.edu.pl}

\newcommand{\fmtitle}{The model of closed-loop control by thermostats: properties and numerical simulations}
\newcommand{\fmabstract}{
A closed-loop control of a reaction-diffusion type process is introduced. The control system consist of a finite number of control and measurement devices. The measurement devices collect information about the current state of the process. The control devices influence the process, responding to data obtained from the measurement devices. Each control device takes into account the data from all measurement devices. The rule of accounting the data from measurement devices by a single control device involves defining suitable weights for each pair of one control device and one measurement device. A weight reflects how important is a given measurement device to a given control device. The aim of this control system is to bring the process possibly close to a user defined reference state or trajectory.

We are interested in a situation where the user can adjust the control system by choice of the control and measurement devices and the weights. For this reason, one of the aims of the preset work is to study the behavior under perturbations of these elements for the mathematical model realizing the above control concept. Moreover, we formulate and justify results concerning existence and uniqueness of solutions of the investigated model. Finally, numerical prototypes illustrating properties of the model are presented. }
\newcommand{\fmkeywords}{closed-loop control\sep thermostats\sep stability analysis\sep numerical simulations}
\newcommand{\fmmsc}{35-04\sep 35A02\sep 35B20\sep 35K10\sep 35K58\sep 35Q93} 
\newcommand{\fmdate}{May 2013}

\newcommand{\fmacknowledgement}{
The author would like to thank {\sc Marek Niezg{\'o}dka} from Interdisciplinary Centre for Mathematical and Computational Modelling, University of Warsaw, for several encouraging conversations concerning the subject of this paper.

As a Ph.D. student in \textit{Ph.D. Programme: Mathematical Methods in Natural Sciences}, the author obtained support from the \textit{International Ph.D. Projects} program of Foundation for Polish Science financed within the \textit{Innovative Economy Operational Programme} 2007-2013 funded by European Regional Development Fund.

The the results of the present paper will be included in author's Ph.D. thesis (in preparation).
}

\begin{titlepage}

\title{\fmtitle}
\author{{\sc \fmauthor}\\ \\ \fminstitution\\ \fmaddress \\ \\ \fmemail}
\date{\fmdate}

\maketitle
\thispagestyle{empty}

\begin{abstract}
\fmabstract
\end{abstract}


\textit{2010 MSC:} \fmmsc

\vspace{15pt}

\begin{center}
 PREPRINT
\end{center}

\begin{figure}[b]
\begin{center}
\includegraphics[width=12cm]{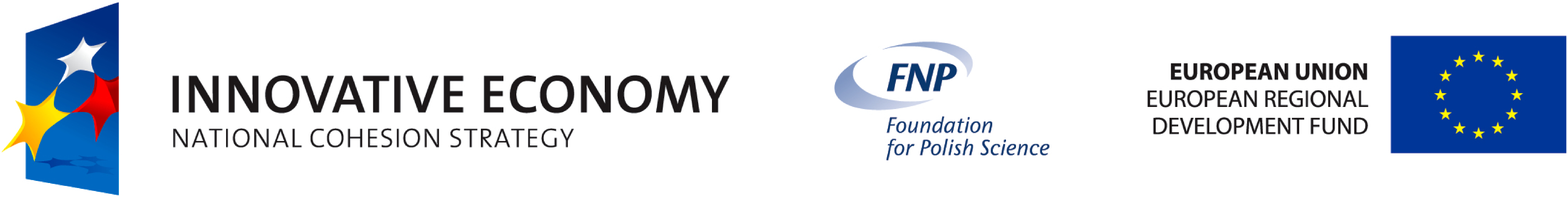}
\end{center} 
\end{figure}

\end{titlepage}

\section{Introduction}\setcounter{equation}{0}
We would like to present and analyze a closed loop control by a finite number of control and measurement devices for a process governed by a semilinear reaction-diffusion equation. The measurement devices located in the domain of the process gather the data about the current state of the process. The control devices supply energy into the process domain basing on the information obtained from the measurement devices. The purpose of the control system is to keep the evolution of the process possibly close to some prescribed evolution demanded by the user. 

For the above control concept we take into consideration the following mathematical model:
\begin{equation}
\left\{\begin{array}{ll}
y_t(x,t) - D\Delta y(x,t) = f(y(x,t)) + \sum_{j=1}^J g_j(x)\kappa_j(t) & \textrm{on $Q_T$}\\
\frac{\partial y}{\partial n} = 0 & \textrm{on $\partial\Omega\times (0,T)$}\\
y(0,x)=y_0(x) & \textrm{for $x\in\Omega$}
\end{array}\right.
\label{eqnmain1}\end{equation}
together with
\begin{equation}
\left\{\begin{array}{ll}
\beta_1\kappa_1'(t) + \kappa_1(t) = W_1(y(\,.\,,t),y^*(\,.\,,t))& \textrm{on $[0,T]$}\\
\vdots &\vdots\\
\beta_J\kappa_J'(t) + \kappa_J(t) = W_J(y(\,.\,,t),y^*(\,.\,,t))& \textrm{on $[0,T]$}\\
\kappa_j(0)=\kappa_{j0}\in\R & \textrm{for $j=1,\ldots,J$} 
\end{array}\right.
\label{eqnmain2}\end{equation}
where $Q_T=\Omega\times (0,T)$, $\Omega$ is a bounded domain in $\R^d,d\geq 1$, with sufficiently regular boundary and the unknown is the sequence $(y,\kappa_1,\ldots,\kappa_J)$ for $y\colon Q_T\rightarrow\R$, $\kappa_j\colon [0,T]\rightarrow\R$. The term $f\colon\R\rightarrow\R$ is a given nonlinearity, $T,D,\beta_1,\ldots,\beta_J>0$ are given constants, $y^*\colon Q_T\rightarrow\R$, $g_j\colon\Omega\rightarrow\R$ and the functionals $W_j$ are of the form
\begin{equation}
W_j(y(\,.\,,t),y^*(\,.\,,t))=\sum_{k=1}^K \alpha_{jk} w_k\left( \int_\Omega h_k(x) (y(x,t) - y^*(x,t)) dx )\right)
\label{eqnmain3}\end{equation}
where $\alpha_{jk}\in\R$, $w_k\colon\R\rightarrow\R$ and $h_k\colon\Omega\rightarrow\R$.

In system (\ref{eqnmain1}) - (\ref{eqnmain3}), the function $y^*$ describes \emph{a reference trajectory} --- the purpose of the introduced control system is to stabilize the reaction-diffusion process as close as possible to the reference trajectory $y^*$. $h_k$ functions characterize \emph{the measurement devices} collecting the data on the current state of the process. \emph{The measurement value} is the term $\int_\Omega h(y-y^*)\:dx$ entering the right hand side of (\ref{eqnmain3}). 
$g_j$ are constant in time functions characterizing \emph{the control devices}. The actions of control devices are moderated in time by functions $\kappa_j$. Thus $\kappa_j$ can be understood as \emph{the signal} for $j$-th control device, produced by some \emph{signal generator}. $\kappa_j$ are modeled with ODEs in (\ref{eqnmain2}) meaning that the changes of signal in time are continuous. The measurement values returned by the measurement devices are processed by functions $w_k$. The processed measurement data are taken into account by the generator of signal for $j$-th control device with \emph{weights} $\alpha_{jk}$, $k=1,\ldots,K$. For the $w_k$ functions, a natural example is $w_k(s)=-sgn(s)$. In this case, function $w_k$ returns simple information understood by the signal generators as ,,heat'' or ,,cool'', depending on whether the value of $k$-th measurement indicates that the process values exceed the reference values or are below them. Hence, the $w_k$ functions can be understood as functions describing a switching mechanism working in the system. We will call $w_k$ \emph{the switching functions}.

It is common that in the real life situations the user of the control system can decide about amount and location of the control devices as well as the measurement devices and can choose the corresponding weights. Thus in the present work we will understand \emph{the control} as the set of all control devices, measurement devices and weights appearing in  (\ref{eqnmain1}) - (\ref{eqnmain3}), i.e. the control is $\left( {g_j},{h_k},{\alpha_{jk}}\right)_{j=1,\ldots,J}^{k=1,\ldots,K}$.

In the preset work, the closed-loop control concept in system (\ref{eqnmain1}) - (\ref{eqnmain3}) will be called \emph{the control by thermostats} or \emph{the thermostats control system}. 
The idea of control by thermostats of processes governed by parabolic PDEs was introduced in \cite{gs1}, \cite{gs2}. There, a linear heat flow was controlled by thermostats. 
Later the concept of thermostats was used to control many other types of processes. In this context, phase transition processes described by various versions of the Stefan model were considered, see e.g. \cite{fh}, \cite{hns}, \cite{cgs}. in \cite{cc}, \cite{gj1}, \cite{dn} processes described by semilinear equations were controlled by thermostats. In \cite{gj2} a control by thermostats for linear heat flow was considered. 

Not only the controlled process varies in the mentioned works. The control by thermostats also has its variants. One of the point where the differences in the \tcb{thermostats control system} can occur is the placement of actions of the control devices. In all indicated references, except for \cite{dn}, the control devices are acting through the boundary of the process. In \cite{dn} and the present work the control devices create a power spot distributed in the domain. Also, various versions of the switching mechanism working in the \tcb{thermostats control system} can be found in the literature. A frequently encountered case is that hysteresis in the work of the switching mechanism is assumed to be present. See \cite{hns}, \cite{cc}, \cite{gj2} for an applications of the so-called relay switch hysteresis or \cite{fh}, \cite{cgs}, \cite{cc}, \cite{gj1} for the Preisach hysteresis model. In \cite{hns}, \cite{cgs} or \cite{dn} the case of no hysteresis effects in the switching mechanism was addressed. In the present work, we also don't assume hysteresis effects. 

The version of the \tcb{thermostats control system} investigated in this work is very similar to that in \cite{dn} or one of the cases considered in \cite{hns}. However, in comparison to those works, we make additional assumptions for the switching functions to get stronger results. 

An example application of control by thermostats  can be a hyperthermia cancer therapy. In the context of hyperthermia, the  authors of \cite{BSW02} describe an attempt to construct a real closed-loop control system with the thermostats concept similar to ours. Moreover, the concept of distributed power spots created by the control devices is consistent with this application. Namely, the power spots inside the domain can be understood as effects of deeply penetrating media, such as electromagnetic waves or ultrasounds. These media can be utilized as heating media in hyperthermia treatment --- see \cite{Zee02}, \cite{CDSD10}. 
Other possible application for the subject control concept is the area of controlling the free boundary problems. As already mentioned, a control by thermostats for the Stefan type state equations was taken into account in the literature. 
But the potential application fields are much more numerous. The numerical prototypes in the further parts of the present work suggest that system (\ref{eqnmain1}) - (\ref{eqnmain3}) has interesting properties. Namely, its solutions are resistant to perturbations of the initial condition and it is well suited for the task of preserving unstable states of the process. 
These are natural properties of the \tcb{thermostats control system} being a direct consequence of the closed-loop control idea. In model (\ref{eqnmain1}) - (\ref{eqnmain3}) the actions of the control devices are not \textit{a priori} prescribed. The \tcb{thermostats control system} tracks the evolution of the process in real time and brings corrections to the actions of the control devices, if necessary. This makes the subject control concept applicable for all the engineering fields where the immunity to the perturbations of the initial condition or a good behavior in unstable situations are crucial.

The main objective of our present research is the problem of optimal choice of locations of control and measurement devices in model (\ref{eqnmain1}) - (\ref{eqnmain3}). 
More precisely, assume that functions $\sigma_g,\sigma_h:\R^d\rightarrow\R$ and points $\hat{\upsilon}_1,\ldots,\hat{\upsilon}_J$ and $\hat{\eta}_1,\ldots,\hat{\eta}_K$ in $\R^d$ are given. We are interested in a situation where the control and measurement devices are given by formula $g_j(x) := \sigma_g(x - \hat{\upsilon}_{j})|_\Omega$ and $h_j(x) := \sigma_h(x - \hat{\eta}_{k})|_\Omega$ for $j=1,\ldots,J$, $k=1,\ldots,K$. 
Having fixed $\sigma_g$ and $\sigma_h$ and weights $\alpha_{j,k}$, the set of points $\hat{\upsilon}_1,\ldots,\hat{\upsilon}_J$ and $\hat{\eta}_1,\ldots,\hat{\eta}_K$ can be understood as \emph{a control parameter}, since it determines the control uniquely. An example choice of $\sigma_g$ and $\sigma_h$ can be $\sigma_h(x) = \sigma_g(x) = \mathbf{1}_{B(0,r)}(x)$ for certain $r>0$. The problem is to choose a control parameter in an optimal manner assuming that $\sigma_g$, $\sigma_h$ and $\alpha_{j,k}$ are fixed. The optimality criterion can be for example to minimize the cost functional given by $\int_{t_1}^T \int_\Omega \hnorm{y-y^*}\:dx\:dt$ for certain $0\leq t_1<T$. This cost functional is simple but reflects the idea of keeping the process close to the reference trajectory $y^*$.

To our knowledge, mathematical analysis of the above optimal localization problem was not performed so far, neither for model (\ref{eqnmain1}) - (\ref{eqnmain3}) nor for other mentioned models involving control by thermostats. Many other questions were posed, including the existence and uniqueness of solutions (see  \cite{gs1}, \cite{gs2}, \cite{hns}, \cite{cgs}, \cite{cc}), the existence of periodic solutions (see \cite{gj1}, \cite{gj2}) or the existence of a global attractor (see \cite{gj1}). 
We have encountered in the literature only one type of optimization problems for models involving control by thermostats. It concerned the problem of choosing the optimal hysteresis law, for the variant of control by thermostats where a switching mechanism with hysteresis was considered --- see e.g. \cite{fh}.

The total material necessary for the complete treatment of the optimal location problem is lengthy --- it covers variety of questions, from the stability analysis, through formulating the suitable optimal control problem and deriving the optimality conditions, up to choosing and implementing a proper optimization method. It would be too much to put all these topics into a single paper. Hence we have decided to contain here only a part of them.

The main objective of the present paper is to give estimates and stability analysis for solutions of system (\ref{eqnmain1}) - (\ref{eqnmain3}). The stability analysis concerns, in particular, the behavior of the subject system under perturbations of the control. This forms background for further stages of our research, concerning directly the optimal location problem. Also, the results on the existence and uniqueness for system (\ref{eqnmain1}) - (\ref{eqnmain3}) are justified in the present work. In addition, we present some numerical prototypes illustrating behavior of system (\ref{eqnmain1}) - (\ref{eqnmain3}).

The part of our research concerning directly the problem of choosing optimal locations for control and measurement devices is mostly completed now. Its results will be described in detail in our next paper (in preparation).

The matter of the current work is split into two parts --- one of them concerns rigorous mathematical analysis, the other focuses on numerical computations.

The rigorous analysis part is contained in Section \ref{sec:CP.1.continuity}. We prove theorems concerning estimates and stability in the suitable space for solutions of (\ref{eqnmain1}) - (\ref{eqnmain3}). The stability is shown w.r.t. both control and initial condition. These estimate and stability results are the key theorems in the present work. Moreover, the existence of solutions for Lipschitz continuous bounded switching functions is shown, by means of the Schauder fixed-point theorem. Next, we use the result concerning estimates for solutions of (\ref{eqnmain1}) - (\ref{eqnmain3}) to generalize the existence result to the class of unbounded Lipschitz switching functions. The uniqueness of solutions is inferred by the result on stability of solutions of (\ref{eqnmain1}) - (\ref{eqnmain3}) w.r.t. initial condition. \tcr{We conclude Section \ref{sec:CP.1.continuity} by proving the weak subsequential stability of system (\ref{eqnmain1}) - (\ref{eqnmain3}) when a sequence of controls converging only weakly is given.}

The numerical computations part is the subject of Section \ref{sec:CP.numexamples}. In  Section \ref{sec:CP.numexamples} we intend to illustrate certain properties of the control by thermostats which make this control concept valuable from the point of view of applications. We will describe a numerical experiment concerning the behavior of the subject control concept in a situation when the reference trajectory $y^*$ is assumed to be an unstable state of the system. What is emphasized there is the simplicity of adjustment of the \tcb{thermostats control system} which is sufficient to make this system efficient for the task of preserving the unstable state. Next, we present some simulations which are intended to test the behavior of the \tcb{thermostats control system} under perturbations of the initial condition. The obtained results suggest that there exists quite strong immunity of the system to this kind of perturbations. We conclude this part of the paper with an experiment for comparing the evolution of the process with two various amounts of the control and measurement devices of the same size. The size of the devices can be understood as $\textrm{diam}(\textrm{supp}(g_j))$ or $\textrm{diam}(\textrm{supp}(h_k))$ respectively. The observation that could be expected is made: smaller amount of the devices results in decreased quality of output of the \tcb{thermostats control system}. This brings us to the question on optimal locations of the control and measurement devices.

Before we proceed to the analytical part of our work (Section \ref{sec:CP.1.continuity}) and to describing the numerical simulations (Section \ref{sec:CP.numexamples}), in Section \ref{sec:CP.1.main} we give some preliminaries necessary for precise setting of the problem for further treatment.

\section{Setting the problem --- assumptions and the state operator}\label{sec:CP.1.main}\setcounter{equation}{0}

In this section we give the preliminaries concerning the mathematical assumptions and definitions which are in use in the further parts of the work. 

However, before we pass to the content related to our model directly, it is necessary to introduce some general notation. Assuming, that the domain $\Omega\subset\R^d$ and a measurable set $E\subseteq\R^d$ are given, we define:\\
\renewcommand{\arraystretch}{1.25}
\begin{longtable}{lcp{0.75\textwidth}}
$\xnorm{\,.\,}{X}$					& --- & the norm of a given Banach space $X$,\\
$\left( \,.\,,\,.\, \right)_{H}$					& --- & the scalar product of a given Hilbert space $H$,\\
$\left<\,.\, ,\,.\, \right>_{\dual{X},X}$		&	--- & the natural pairing between $\dual{X}$ and $X$, for a given Banach space $X$; the first argument stands for the element of $\dual{X}$,\\
$\xnorm{\,.\,}{p}$					& --- & the norm of the Lebesgue space $\Lp{p}{\Omega}$, $p\in [1,\infty]$,\\
$\xnorm{\,.\,}{X,q}$				& --- & the norm of the Bochner space $\eLp{q}{X}$, where $X$ is a~Banach space, $q\in [1,\infty]$,\\
$\xnorm{\,.\,}{p,q}$				& --- & the norm of the Bochner space $\eLp{q}{\Lp{p}{\Omega}}$,\\
$\left<\,.\, ,\,.\, \right>$		&	--- & the natural pairing between $\dual{\Ha{1}}$ and $\Ha{1}$; the first argument stands for the element of $\dual{\Ha{1}}$,\\
$\xnorm{\,.\,}{p,E}$					& --- & the norm of the Lebesgue space $\Lp{p}{E}$, $p\in [1,\infty]$.\\
\\
\end{longtable}
\renewcommand{\arraystretch}{1}
In addition, we don't want to bother with separate notation for vector valued functions. Hence we denote the standard norm of $\left(\Lp{p}{\Omega}\right)^d$ or $\eLp{q}{\left(\Lp{p}{\Omega}\right)^d}$ simply as $\xnorm{\,.\,}{p,\mathbb{E}}$ or $\xnorm{\,.\,}{p,q}$, respectively. Similarly, the standard scalar product in $\left(\Lp{2}{\mathbb{E}}\right)^n$ will be denoted as $\left(\,.\,,\,.\, \right)_{\Lp{2}{\mathbb{E}}}$.

Moreover, we denote
\begin{equation}\begin{array}{rcl}
X^0 &=& \Lp{2}{\Omega}\times\R^J\\
X^1 &=& \Lp{2}{Q_T}\times\left(\Lp{2}{0,T}\right)^J\\
X^2 &=& \left\{ (y,\kappa_1,\ldots,\kappa_J):\ y\in\eLp{\infty}{\Lp{2}{\Omega}},\ \nabla y\in \left(\Lp{2}{Q_T}\right)^d, \right.\\
		&&	\left.\ y'\in\eLp{2}{\dual{\Ha{1}}} \textrm{ and } \right.\\
		&&	\left.\ \kappa_j\in\Lp{\infty}{0,T},\ \kappa_j'\in\Lp{2}{0,T} \textrm{ for } j=1,\ldots,J \right\}\\
\end{array}\end{equation}
The derivatives $y'$ and $\kappa_j'$ in the definition of the space $X^2$ is understood in the sense of vector-valued distributions and $\nabla y$ refers to the vector of the weak partial derivatives of $y$ w.r.t. the space variables.

Now, we introduce the assumptions concerning system (\ref{eqnmain1}) - (\ref{eqnmain3}) that will be required for our considerations:
\begin{asmlisti}
		\item\asmlabeli{asm02} $T>0$ and $\Omega\subset\R^d$ is a bounded domain, such that the Rellich-Kondrachov theorem for the embedding $\Ha{1}\subset\Lp{2}{\Omega}$ is valid, 
		\item\asmlabeli{asm03} $K$ and $J$ are given positive natural numbers,
		\item\asmlabeli{asm04} $f$ is globally Lipschitz continuous with a constant $L$ and $f(0)=:f_0$,
		\item\asmlabeli{asm06} $w_k$ is globally Lipschitz continuous with a constant $L_k$ and $w_k(0) =:w_{k0}$, for all $k=1,\ldots,K$,
		\item\label{asm07} $D>0$, $\beta_j>0$ for all $j=1,\ldots,J$,
		\item\label{asm10} $y^*\in\eLp{\infty}{\Lp{2}{\Omega}}$,
		\item\asmlabeli{asm12} $y_0\in\Lp{2}{\Omega}$, $\kappa_{j0}\in\R$ for $j=1,\ldots,J$.
\end{asmlisti}
For assumption \asmstylei{asm02}, a bounded domain satisfying the cone condition is sufficient (see \cite[Th. 6.3.]{af}).

\tcr{As mentioned in the introduction}, we will consider the set of control devices, measurement devices and corresponding weights to be the user-eligible components. Hence  the sequence $({g_j},{h_k},{\alpha_{jk}})_{j=1,\ldots,J}^{k=1,\ldots,K}$, where the particular components correspond to the quantities appearing in (\ref{eqnmain1}) - (\ref{eqnmain3}),  will be called \emph{a control}. Moreover, we define the following spaces:
\begin{equation}\begin{array}{llll}
U = U_g\times U_h\times U_\alpha,\quad & U_g = \left(\Lp{2}{\Omega}\right)^J, & U_h = \left(\Lp{2}{\Omega}\right)^K, & U_\alpha = \R^{KJ}\\
\end{array}
\label{eqn206}\end{equation}
$U$ will be called \emph{the control space}. We equip it with standard product topology and scalar product. Arbitrary sufficiently regular control  $({g_j},{h_k},{\alpha_{jk}})_{j=1,\ldots,J}^{k=1,\ldots,K}$ can be interpreted as an element of $U$ and \textit{vice versa} --- arbitrary element $\hat{u}\in U$ gives a control. Hence, for reasons of convenience, let us also develop notation concerning the elements of the control space $U$ . For a given element $\hat{u}\in U$ we denote the coordinates of $\hat{u}$ in the following way:
\begin{displaymath}
\hat{u}=(\hat{u}_{g_j},\hat{u}_{h_k},\hat{u}_{\alpha_{jk}})_{j=1,\ldots,J}^{k=1,\ldots,K}\qquad 
\end{displaymath}
\begin{displaymath}
\textrm{where}\quad(\hat{u}_{g_1}\ldots,\hat{u}_{g_J})\in U_g,\quad (\hat{u}_{h_1},\ldots,\hat{u}_{h_k})\in U_h,\quad (\hat{u}_{\alpha_{j,k}})_{j=1,\ldots,J}^{k=1,\ldots,K}\in U_\alpha
\end{displaymath}

The following definition of solutions for system (\ref{eqnmain1}) - (\ref{eqnmain3}) will be utilized in the present paper: 
\begin{definition}
$(y,\kappa_1,\ldots,\kappa_J)$ is a weak solution to system (\ref{eqnmain1}) - (\ref{eqnmain3}) if:
\begin{enumerate}[(a)]
\item $(y,\kappa_1,\ldots,\kappa_J)\in X^2$,
\item\label{def04.06} $y(0) = y_0$ in $\Lp{2}{\Omega}$ and $\kappa_j(0)=\kappa_{j0}$ for $j=1,\ldots,J$,
\item\label{def04.08} $\int_0^T \left<y',\phi \right> + D(\nabla y,\nabla\phi)_{\Lp{2}{\Omega}} + ( - f(y) - \kappa_1g_1 - \ldots - \kappa_Jg_J,\phi)_{\Lp{2}{\Omega}}\:dt = 0$\\
for all $\phi \in \eLp{2}{\Ha{1}}$,
\item\label{def04.10} $\int_0^T \left(\beta_j\kappa_j' + \kappa_j - W_j(y,y^*)\right)\xi \:dt = 0$ for all $\xi\in\Lp{2}{0,T}$, for $j=1,\ldots,J$.
\end{enumerate}
\label{def04}\end{definition} 
The spaces $\Ha{1}\subset\Lp{2}{\Omega}\subset\dual{\Ha{1}}$ form so called evolution triple (see \cite[Chap. 23.4]{zei2a}). Hence, by \cite[Prop. 23.23]{zei2a}, if $(y,\kappa_1,\ldots,\kappa_J)\in X^2$ then $y\in\Ce{\Lp{2}{\Omega}}$. This makes point (\ref{def04.06}) in the above definition meaningful. Verifying the existence and the uniqueness of solutions of system (\ref{eqnmain1}) - (\ref{eqnmain3}) will be one of the aims of the next section.

\begin{remark}
Concerning the weights $\alpha_{j,k}$ in (\ref{eqnmain3}), one can expect an assumption that $\alpha_{j,k}$ are nonnegative and summable to unity over $k=1,\ldots,K$, for all $j=1,\ldots,J$. But this assumption doesn't play any role in our considerations, hence we don't impose it and allow $\alpha_{jk}$ to be arbitrary real numbers. This is reflected in the structure of the control space $U$.
\end{remark}

\section{Properties of the system}\label{sec:CP.1.continuity}\setcounter{equation}{0}

In this section we derive estimates for solutions of system (\ref{eqnmain1}) - (\ref{eqnmain3}) and investigate stability of  (\ref{eqnmain1}) - (\ref{eqnmain3}) w.r.t. both control and initial condition. \tcr{ We also prove the weak subsequential stability of (\ref{eqnmain1}) - (\ref{eqnmain3}) when the control space is considered with its weak topology.}

The purpose of this analysis is twofold. First, we are interested in the well-posedness of system (\ref{eqnmain1}) - (\ref{eqnmain3}), i.e. we want to infer the existence and uniqueness of solutions to this model. The estimates that will be derived and the stability w.r.t. initial condition will be essential for justification of the theorem on existence and uniqueness of the weak solutions of system (\ref{eqnmain1}) - (\ref{eqnmain3}). Second, the results concerning stability w.r.t. the control variable are useful from the point of view of the optimal control theory. These results will be useful for our further research.

Before the announced estimates for solutions and stability analysis, we start with the existence result for the case when $w_k$ satisfy some restrictions additional to those demanded in assumption \asmstylei{asm06}:
\begin{theorem}
Assume that general assumptions \asmstylei{asm02} - \asmstylei{asm12} hold and $\left(g_j,h_k,\alpha_{jk} \right)_{j=1,\ldots,J}^{k=1,\ldots,K}\in U$. Assume moreover that functions $w_k$ are bounded for $k=1,\ldots,K$. Then system (\ref{eqnmain1}) - (\ref{eqnmain3}) has a weak solution.
\label{theorem04.0}\end{theorem}
Before passing to the proof, we give a short bibliographical remark. Mathematical models in \cite{hns} and \cite{dn} utilize a control by thermostats that is very similar to ours, or in some sense more general since they don't assume that the switching functions are Lipschitz continuous. The proof of existence in \cite{hns} and \cite{dn} consists in use of generalized Kakutani fixed-point theorem. Nevertheless, some differences between our work occurs and the models considered in \cite{hns} and \cite{dn} so our model cannot be viewed as a particular case of the models in \cite{hns} and \cite{dn}. Hence we decided to present the proof of existence here. We will prove the assertion with use of the Schauder fixed-point theorem, which is less general that the generalized Kakutani theorem but sufficient for our purposes.

\noindent\begin{proof}
For the sake of brevity, we will present the proof for $J=K=1$, $\alpha_{1,1}=1$. The proof for the general case follows the same lines.

\newcommand{\solspace}{\mathbb{W}}

Denote $\solspace_y(0,T) = \left\{ y\in\eLp{\infty}{\Lp{2}{\Omega}},\:\nabla y\in\Lp{2}{Q_T},\:y'\in\eLp{2}{\dual{\Ha{1}}}\right\}$ and $\solspace_\kappa(0,T) = \left\{ \kappa,\kappa'\in\Lp{2}{0,T}\right\}$. Consider the following equations:
\begin{equation}\left\{\begin{array}{ll}
 y_t(x,t) - D\Delta y(x,t) = f(y(x,t)) + G(x)k(t) & \textrm{on $Q_T$}\\
 \frac{\partial y}{\partial n} = 0 & \textrm{on $\partial \Omega\times(0,T)$}\\
 y(0) = y_0 & \textrm{on $\Omega$}
\end{array}\right.
\label{eqn1existence}\end{equation}
\begin{equation}\left\{\begin{array}{ll}
 \beta \kappa'(t) + \kappa(t) = W(t) & \textrm{on $[0,T]$}\\
 \kappa(0) = \kappa_{0}
\end{array}\right.
\label{eqn2existence}\end{equation}
\begin{equation}\begin{array}{ll}
 W(t) = w\left( \int_\Omega H(x)\left(Y(x,t) - y^*(x,t)\right)\:dx \right) & \textrm{a.e. on $[0,T]$}\\
\label{eqn3existence}\end{array}\end{equation}
where $k\in\Lp{2}{0,T}$, $W\in C[0,T]$ and $Y\in \Ce{\Lp{2}{\Omega}}$ are given, $G,H\in\Lp{2}{\Omega}$, $\beta>0$, $\kappa_0\in\R$ and $D$, $f$ are as in the assumptions of the theorem.

By solution of (\ref{eqn1existence}) we understand a function $y\in\solspace_y(0,T)$ that satisfies $y(0)=y_0$ and
\begin{equation}
\int_0^T \left<y',\phi \right> + D(\nabla y,\nabla\phi)_{\Lp{2}{\Omega}} + ( - f(y) - Gk,\phi)_{\Lp{2}{\Omega}}\:dt = 0 
\label{soldef1existence}\end{equation}
for all $\phi \in \eLp{2}{\Ha{1}}$. By solution of (\ref{eqn2existence}) we mean a function $\kappa\in\solspace_\kappa(0,T)$ that satisfies $\kappa(0) = \kappa_0$ and 
\begin{equation}
 \int_0^T \left(\beta\kappa' + \kappa - W\right)\xi \:dt = 0
\label{soldef2existence}\end{equation}
for all $\xi\in\Lp{2}{0,T}$. Since $\solspace_y(0,T)\subseteq \Ce{\Lp{2}{\Omega}}$ (see \cite[Prop. 23.23]{zei2a}) and $\solspace_\kappa(0,T)\subseteq \Cont{[0,T]}$ (by Sobolev embedding theorem, see \cite[Th. 4.12]{af}), the initial condition for solutions of both (\ref{eqn1existence}) and (\ref{eqn2existence}) are well defined.

It is well known that the solutions of equations (\ref{eqn1existence}) and (\ref{eqn2existence}) exist and are unique. 

It is also a consequence of known estimates that for $y^1$ and $y^2$ being two solutions of  (\ref{eqn1existence}) corresponding to $k^1\in\Lp{2}{0,T}$ and $k^2\in\Lp{2}{0,T}$ respectively there holds
\begin{equation}
 \xnorm{y^1 - y^2}{\solspace_y(0,T)} \ \leq\ C_1\ \xnorm{G(x)(k^1(t) - k^2(t))}{2,2} \ \leq\ C_2 \xnorm{k^1 - k^2}{\Lp{2}{0,T}}
\label{ineq1existence}\end{equation}
This is rather standard estimate thus we don't prove it here.

Even easier it can be shown that for $\kappa^1$ and $\kappa^2$ being two solutions of (\ref{eqn2existence}) corresponding to $W^1\in\Lp{2}{0,T}$ and $W^2\in\Lp{2}{0,T}$ respectively we have
\begin{equation}
 \xnorm{\kappa^1 - \kappa^2}{\solspace_\kappa(0,T)} \ \leq \ C_3 \xnorm{W^1 - W^2}{\Cont{[0,T]}}
\label{ineq2existence}\end{equation}

Now, we define the following operators. $P_1\colon\Lp{2}{0,T}\rightarrow\Ce{\Lp{2}{\Omega}}$ is an operator assigning the solution of (\ref{eqn1existence}) to given $k\in\Lp{2}{0,T}$. It is well defined since the solution (\ref{eqn1existence}) exists in $\solspace_y(0,T)$, is unique and $\solspace_y(0,T)\subseteq \Ce{\Lp{2}{\Omega}}$. $P_2\colon\Ce{\Lp{2}{\Omega}}\rightarrow\Cont{[0,T]}$ assigns $W$ given by formula (\ref{eqn3existence}) to given $Y\in \Ce{\Lp{2}{\Omega}}$. $P_3\colon\Cont{[0,T]}\rightarrow\solspace_\kappa(0,T)$ assigns the solution of (\ref{eqn2existence}) for given $W\in\Cont{[0,T]}$. It is well defined since the solution  of (\ref{eqn2existence}) exists in $\solspace_\kappa(0,T)$ and is unique.

Proving that $P:=P_3\circ P_2\circ P_1$ has a fixed point in $\Lp{2}{0,T}$ is equivalent to proving the assertion of the theorem for the case of $J=K=1$ and $\alpha_{1,1}=1$, by making the following assignments in the system (\ref{eqnmain1}) - (\ref{eqnmain3}): $\kappa_{10}:=\kappa_{0}$, $w_1:=w$, $g_1:=G$, $h_1:=H$, $\beta_1:=\beta$. In other words, we need to prove that there exists $\kappa\in\Lp{2}{0,T}$ such that $\kappa = P_3(W)$, $W = P_2(Y)$, $Y = P_1(\kappa)$.

By (\ref{ineq1existence}) and (\ref{ineq2existence}) the operators $P_1$ and $P_3$ are continuous. By assumption that $w$ is Lipschitz continuous we also verify the continuity of $P_2$. Let $W^1 = P_2(Y^1)$ and $W^2 = P_2(Y^2)$ for given $Y^1,Y^2\in\Ce{\Lp{2}{\Omega}}$. Then:
\begin{displaymath}\begin{array}{rcl}
 \xnorm{W^1-W^2}{\Lp{\infty}{0,T}} &\leq& \Lipfunc{w}\hnorm{\int_\Omega H(x)(Y^1(x,t)-Y^2(x,t))\:dx} \\
  &\leq & \Lipfunc{w}\xnorm{H}{2}\xnorm{Y^1 - Y^2}{2,\infty}
\end{array}\end{displaymath}

Next, recall the assumption that $w$ is bounded. We denote $C_4:=\xnorm{w}{\Lp{\infty}{\R}}$. It is straightforward, that $P_2\colon\Ce{\Lp{2}{\Omega}}\rightarrow\mathbb{N}$ where $\mathbb{N}:=\left\{ W\in \Cont{[0,T]}:\:\xnorm{W}{\Lp{\infty}{0,T}}\leq C_4\right\}$. By linearity of equation (\ref{eqn2existence}) and the estimate (\ref{ineq2existence}) we also get that $P_3|_\mathbb{N}\colon\mathbb{N}\rightarrow\mathbb{M}$ where $\mathbb{M}:=\left\{\kappa\in\solspace_\kappa(0,T):\: \xnorm{\kappa}{\solspace_\kappa(0,T)}\leq C_3C_4 \right\}$. Set $\mathbb{M}$ is nonempty, convex and compact in $\Lp{2}{0,T}$ (by Rellich-Kondrachov theorem, see \cite[Th. 6.3]{af}).

To sum up, we have shown that $P=P_3\circ p_2\circ P_1\colon \Lp{2}{0,T}\rightarrow \mathbb{M}$ where $\mathbb{M}$ is nonempty, convex and compact in $\Lp{2}{0,T}$. Hence $P$ has a fixed point in $\mathbb{M}$ by the Schauder theorem (see Chap. 2.6 in \cite{zei1}).
\end{proof}
The existence for arbitrary $w_k$ satisfying assumption \asmstylei{asm06} and the uniqueness will be a consequence of other theorems which will be proven below.
\begin{theorem}
Let the assumptions \asmstylei{asm02} - \asmstylei{asm12} be fulfilled, let $\hat{u}\in U$ and $(y_0,\kappa_{10},\ldots,\kappa_{J0})\in X^0$. Assume also that $\xnorm{\hat{u}}{U}\leq R^U$ for some $R^U> 0$ and that $\xnorm{(y_0,\kappa_{10},\ldots,\kappa_{J0})}{X^0}\leq R^0$ for some $R^0> 0$. Let $(y,\kappa_1,\ldots,\kappa_J)\in X^2$ be a weak solution of system (\ref{eqnmain1}) - (\ref{eqnmain3}) corresponding to $g_j:=\hat{u}_{g_j}$, $h_k:=\hat{u}_{h_k}$, $\alpha_{j,k}:=\hat{u}_{\alpha_{j,k}}$ and initial condition $(y_0,\kappa_{10},\ldots,\kappa_{J0})$. Then the following estimate holds:
\begin{displaymath}
\xnorm{(y,\kappa_1,\ldots,\kappa_J)}{X^2}\ <\ C
\end{displaymath}
where 
\[C = C(T,\hnorm{\Omega},K,J,L,f_0,L_1,\ldots,L_K,w_{10},\ldots,w_{K0},R^U,R^0,\xnorm{y^*}{2,2},D,\beta_1,\ldots,\beta_J)\] 
and where the appearing quantities are the same as those in the general assumptions referred above.
\label{theorem310}\end{theorem}
\begin{proof}
We test the weak form of the first equation in (\ref{eqnmain1}), see Definition \ref{def04}, by $y(x,s)\mathbf{1}_{(0,t)}(s)$ and obtain (by using the H\"older inequality, the Young inequality and our structural assumptions --- note that $\hnorm{f(s)}\leq \hnorm{f_0} + L\hnorm{s}$):
\begin{equation}\begin{array}{cl}
			&\int_0^t \left<y',y\right> + D\xnorm{\nabla y}{2}^2 \:ds = \int_0^t (f(y),y)_{\Lp{2}{\Omega}}+\sum_{j=1}^J(\kappa_j \hat{u}_{g_j},y)_{\Lp{2}{\Omega}} \:ds\\
\leq	&\int_0^t L\xnorm{y}{2}^2  + \hnorm{f_0}\int_\Omega\hnorm{y}\:dx + \sum_{j=1}^J \hnorm{\kappa_j(s)}\xnorm{\hat{u}_{g_j}}{2}\xnorm{y}{2}\:ds\\
\leq	&\int_0^t (L + \frac{\hnorm{f_0}}{2})\xnorm{y}{2}^2 + \frac{\hnorm{f_0}}{2}\int_\Omega 1\:dx + \frac{J}{2}\xnorm{y}{2} + \frac{1}{2}\sum_{j=1}^J \hnorm{\kappa_j(s)}^2\xnorm{\hat{u}_{g_j}}{2}^2\:ds\\
=			&\int_0^t (L + \frac{\hnorm{f_0}}{2} + \frac{J}{2})\xnorm{y}{2}^2 + \frac{1}{2}\sum_{j=1}^J \xnorm{\hat{u}_{g_j}}{2}^2\hnorm{\kappa_j(s)}^2\:ds + C_1
\end{array}\label{eqn310}\end{equation}
where $C_1 = \frac{1}{2}Tf_0\hnorm{\Omega}$. At the same time, testing the equation for $\kappa_j$ by $\kappa_j(s)\mathbf{1}_{(0,t)}(s)$ and neglecting the $\hnorm{\kappa_j}^2$ term yields:
\begin{equation}\begin{array}{cl}
			&\beta_j\int_0^t \kappa_j '\kappa_j\:ds \leq \int_0^t \sum_{k=1}^K \hat{u}_{\alpha_{jk}}w_k(\int_\Omega \hat{u}_{h_k}(y-y^*)\:dx)\kappa_j\:ds\\
\leq	&\int_0^t \sum_{k=1}^K \left(L_k\hnorm{\hat{u}_{\alpha_{jk}}}\xnorm{\hat{u}_{h_k}}{2}\xnorm{y-y^*}{2}\hnorm{\kappa_j} + \hnorm{w_{k0}}\hnorm{\hat{u}_{\alpha_{jk}}}\hnorm{\kappa_j(s)}\right)\:ds \\
\end{array}\label{eqn312.0}\end{equation}
The firs term appearing in the sum obeys:
\begin{equation}\begin{array}{cl}
			&L_k\hnorm{\hat{u}_{\alpha_{jk}}}\xnorm{\hat{u}_{h_k}}{2}\xnorm{y-y^*}{2}\hnorm{\kappa_j} \leq \frac{1}{2}L_k^2\hnorm{\hat{u}_{\alpha_{jk}}}^2\xnorm{\hat{u}_{h_k}}{2}^2\xnorm{y-y^*}{2}^2 + \frac{1}{2}\hnorm{\kappa_j(s)}^2\\
\leq	&\frac{1}{2}L_k^2\hnorm{\hat{u}_{\alpha_{jk}}}^2\xnorm{\hat{u}_{h_k}}{2}^2\xnorm{y}{2}^2 + \frac{1}{2}\hnorm{\kappa_j(s)}^2 +\\
+			&\frac{1}{2}L_k^2\hnorm{\hat{u}_{\alpha_{jk}}}^2\xnorm{\hat{u}_{h_k}}{2}^2\xnorm{y^*}{2}^2
\end{array}\label{eqn312.1}\end{equation}
The second term in the right hand side of (\ref{eqn312.0}) satisfies:
\begin{equation}\begin{array}{cl}
			&\hnorm{w_{k0}}\hnorm{\hat{u}_{\alpha_{jk}}}\hnorm{\kappa_j(s)} \leq \frac{1}{2}\hnorm{\kappa_j(s)}^2 + \frac{1}{2}w_{k0}^2\hat{u}_{\alpha_{jk}}^2
\end{array}\label{eqn312.2}\end{equation}
Combining (\ref{eqn312.0}) with (\ref{eqn312.1}) and (\ref{eqn312.2}) yields:
\begin{equation}\begin{array}{cl}
			&\int_0^t \kappa_j '\kappa_j\:ds = \beta_j^{-1}\int_0^t \sum_{k=1}^K \hat{u}_{\alpha_{jk}}w_k(\int_\Omega \hat{u}_{h_k}(y-y^*)\:dx)\kappa_j\:ds\\
\leq	&\int_0^t \frac{1}{2\beta_j}\sum_{k=1}^K \left(L_k^2\hnorm{\hat{u}_{\alpha_{jk}}}^2\xnorm{\hat{u}_{h_k}}{2}^2\xnorm{y}{2}^2\right) + \frac{K}{\beta_j}\hnorm{\kappa_j}^2 \:ds + C_{2,j} + C_{3,j}\\
\end{array}\label{eqn312}\end{equation}
where 
\begin{displaymath}\begin{array}{rcl}
C_{2,j} &=& \frac{T}{2\beta_j}\sum_{k=1}^K w_{k0}^2\hat{u}_{\alpha_{jk}}^2 \\
C_{3,j} &=& \frac{1}{2\beta_j} \sum_{k=1}^K \int_0^t L_k^2\hnorm{\hat{u}_{\alpha_{jk}}}^2\xnorm{\hat{u}_{h_k}}{2}^2\xnorm{y^*}{2}^2\\
		&=& \frac{1}{2\beta_j} \xnorm{y^*}{2,2}\sum_{k=1}^K L_k^2\hnorm{\hat{u}_{\alpha_{jk}}}^2\xnorm{\hat{u}_{h_k}}{2}^2
\end{array}\end{displaymath}

The spaces $\Ha{1}\subset\Lp{2}{\Omega}\subset\dual{\Ha{1}}$ form an evolution triple, 
hence  the identities $2\int_0^t\left<y',y\right> = \xnorm{y(\,.\,,t)}{2}^2 - \xnorm{y(\,.\,,0)}{2}^2$ and $2\int_0^t\kappa_j '\kappa_j\ = \hnorm{\kappa_j(t)}^2 - \hnorm{\kappa_j(0)}^2$ hold (see \cite[Prop. 23.23]{zei2a}). 
We can use these identities in the estimates (\ref{eqn310}) and (\ref{eqn312}), substitute $y(.,0)=y_0$, $\kappa_j(0)=\kappa_{j0}$, neglect the gradient term and, after summation of (\ref{eqn310}) and (\ref{eqn312}) for every $j$, obtain:
\begin{equation}\begin{array}{cl}
		&\xnorm{y(.,t)}{2}^2 + \sum_{j=1}^J \hnorm{\kappa_j(t)}^2 \leq \xnorm{y_0}{2}^2 + \left(\sum_{j=1}^J \hnorm{\kappa_{j0}}^2\right)\\
+		&2C_1 + 2\sum_{j=1}^J\left(C_{2,j} + C_{3,j}\right) + \int_0^t C_4\xnorm{y}{2}^2 + C_5\sum_{j=1}^J \hnorm{\kappa_j}^2 \:ds
\end{array}\label{eqn314}\end{equation}
where
\begin{displaymath}\begin{array}{rcl}
C_4 &=& 2L + \hnorm{f_0} + J + \sum_{j=1}^J\sum_{k=1}^K \frac{L_k}{\beta_j}^2\hnorm{\hat{u}_{\alpha_{jk}}}^2\xnorm{\hat{u}_{h_k}}{2}^2\\
C_5 &=& 2K\max_{j=1,\ldots,J}\beta_j^{-1} + \max_{j=1,\ldots,J}\xnorm{\hat{u}_{g_j}}{2}^2\\
\end{array}\end{displaymath}
Now, we can use the integral Gr\"onwall inequality to find that
\begin{equation}\begin{array}{cl}
&\xnorm{y}{2,\infty}^2 + \sum_{j=1}^J \xnorm{\kappa_j}{\Lp{\infty}{0,T}}^2 \leq\\
\leq	& \left(\xnorm{y_0}{2}^2 + \left(\sum_{j=1}^J \hnorm{\kappa_{j0}}^2\right) + 2C_1 + 2\sum_{j=1}^J(C_{2,j} + C_{3,j})\right)\\
\cdot & (1 + T\max\{C_4,C_5\}e^{\max\{C_4,C_5\}})
\end{array}\label{eqn316}\end{equation}
The structure of the constants $C_1,C_{2,j},C_{3,j},C_4,C_5$ guarantees that the right hand side of the above depends only on the quantities stated in the assertion of the theorem.

Still, to complete the proof we need to estimate the terms $\xnorm{\nabla y}{2,2}$, $\xnorm{y'}{\dual{\Ha{1}},2}$ and $\xnorm{\kappa_j'}{\Lp{2}{0,T}}$. For estimating the gradient term, we again use inequality (\ref{eqn310}) with $t=T$, neglecting the time derivative term:
\begin{equation}\begin{array}{cl}
			&D\xnorm{\nabla y}{2,2}^2 \leq \int_0^T \frac{1}{2}C_4 \xnorm{y}{2}^2 + \frac{1}{2}C_5\sum_{j=1}^J \hnorm{\kappa_j(s)}^2\:ds + C_1 \\
\leq	&C_6\frac{C_4}{2}\xnorm{y}{2,\infty}^2 + C_6\frac{C_5}{2}\sum_{j=1}^J\xnorm{\kappa_j}{\Lp{\infty}{0,T}}^2 + C_1
\end{array}\label{eqn318}\end{equation}
where $C_6=T$ is a constant appearing in estimating 
the norm of $\Lp{2}{0,T}$ by the norm of $\Lp{\infty}{0,T}$. Next, we use (\ref{eqn316}) to estimate the right hand side of the above inequality in terms of constants $C_1,\ldots,C_6$.

To obtain estimates for the time derivative of $y$, we treat it, as well as the other terms appearing in the weak form of (\ref{eqnmain1}), as functionals on the space $\eLp{2}{{\Ha{1}}}$ and estimate their norm by the definition. This yields:
\begin{equation}\begin{array}{cl}
			&\xnorm{y'}{\dual{\Ha{1}},2}\leq D\xnorm{\nabla y}{2,2} + \xnorm{f(y)}{2,2} + \sum_{j=1}^J\xnorm{\kappa_j \hat{u}_{g_j}}{2,2}\\
\leq	&D\xnorm{\nabla y}{2,2} + \xnorm{\hnorm{f_0} + L\hnorm{y}}{2,2} + \sum_{j=1}^J\xnorm{\kappa_j \hat{u}_{g_j}}{2,2}\\
\leq	&D\xnorm{\nabla y}{2,2} + C_6 L\xnorm{y}{2,\infty}^2 +  C_6 C_5 \sum_{j=1}^J\xnorm{\kappa_j}{\Lp{\infty}{0,T}} + (T\hnorm{\Omega})^{1/2}\hnorm{f_0}
\end{array}\label{eqn320}\end{equation}
and again (\ref{eqn316}) and (\ref{eqn318}) can be applied.

Moreover, by the structure of (\ref{eqnmain2}) we easily derive the estimates for the time derivative of $\kappa_j$ (proceeding as in (\ref{eqn312})):
\begin{equation}\begin{array}{cl}
			&\xnorm{\kappa_j'}{\Lp{2}{0,T}}^2\leq 2\xnorm{\kappa}{\Lp{2}{0,T}}^2 + 2\xnorm{\sum_{k=1}^K \hat{u}_{\alpha_{jk}}w_k(\int_\Omega \hat{u}_{h_k}(y-y^*)\:dx)}{\Lp{2}{0,T}}^2\\
\leq	& 2\xnorm{\kappa_j}{\Lp{2}{0,T}}^2 + 2 C_4\xnorm{y}{2,2}^2 + 4C_{2,j} + 4C_{3,j}\\
\end{array}\label{eqn322}\end{equation}
and, as before, apply (\ref{eqn316}) to the right hand side.

Altogether, (\ref{eqn316}) - (\ref{eqn322}) guarantee that all the investigated quantities can be estimated in terms of the constants $C_1,C_{2,j},C_{3,j},C_4,C_5,C_6$, which depend at most on the quantities stated in the assertion of the theorem.
\end{proof}

\begin{theorem}
Let the assumptions \asmstylei{asm02} - \asmstylei{asm12} be fulfilled, let $\hat{u}^1,\hat{u}^2\in U$ and 
\begin{displaymath}
(y_0^1,\kappa_{10}^1,\ldots,\kappa_{J0}^1),\ (y_0^2,\kappa_{10}^2,\ldots,\kappa_{J0}^2)\ \in\  X^0
\end{displaymath}
Assume also that $\xnorm{\hat{u}^i}{U}\leq R^U$ for some $R^U> 0$ and that $\xnorm{(y_0^i,\kappa_{10}^i,\ldots,\kappa_{J0}^i)}{X^0}\leq R^0$ for some $R^0> 0$, for $i=1,2$. Let $(y^i,\kappa_1^i,\ldots,\kappa_J^i)\in X^2$ be a weak solution of system (\ref{eqnmain1}) - (\ref{eqnmain3}) corresponding to $g_j:=\hat{u}_{g_j}^i$, $h_k:=\hat{u}_{h_k}^i$, $\alpha_{j,k}:=\hat{u}_{\alpha_{j,k}}^i$ and initial condition $(y_0^i,\kappa_{10}^i,\ldots,\kappa_{J0}^i)$, for $i=1,2$. Denote $y=y^1 - y^2$, $\kappa_j = \kappa_j^1 - \kappa_j^2$, $\hat{u} = \hat{u}^1 - \hat{u}^2$, $y_0 = y_0^1 - y_0^2$ and $\kappa_{j0} = \kappa_{j0}^1 - \kappa_{j0}^2$. Then:
\begin{displaymath}
\xnorm{(y,\kappa_1,\ldots,\kappa_J)}{X^2}\ <\ C\left(\xnorm{\hat{u}}{U}^2 + \xnorm{(y_0,\kappa_{10}\ldots,\kappa_{J0})}{X^0}^2\right)^{1/2}
\end{displaymath}
where 
\[C = C(T,\hnorm{\Omega},K,J,L,f_0,L_1,\ldots,L_K,w_{10},\ldots,w_{K0},R^U,R^0,\xnorm{y^*}{2,2},D,\beta_1,\ldots,\beta_J)\] 
and where the appearing quantities are the same as those in the general assumptions referred above.
\label{theorem313}\end{theorem}
\begin{proof}
Executing the subtraction by sides of the corresponding systems of equations infers that $(y,\kappa_{1},\ldots,\kappa_{J})$ satisfies (in the weak sense, analogous to that given in the Definition \ref{def04}) the following system of equations:
\begin{equation}
\left\{\begin{array}{ll}
y_t - D\Delta y = f(y^1) - f(y^2)+ \sum_{j=1}^J \left(\hat{u}_{g_j}^1\kappa_j^1 - \hat{u}_{g_j}^2\kappa_j^2 \right) & \textrm{on $Q_T$}\\
\frac{\partial y}{\partial n} = 0 & \textrm{on $\partial\Omega\times (0,T)$}\\
y(0,x)=y_0(x) & \textrm{for $x\in\Omega$}
\end{array}\right.
\label{eqn330}\end{equation}
together with
\begin{equation}
\left\{\begin{array}{ll}
\beta_1\kappa_1' + \kappa_1 = W_1(y^1,y^*) - W_1(y^2,y^*)& \textrm{on $[0,T]$}\\
\vdots &\vdots\\
\beta_J\kappa_J' + \kappa_J = W_J(y^1,y^*) - W_J(y^2,y^*)& \textrm{on $[0,T]$}\\
\kappa_j(0)=\kappa_{j0}\in\R & \textrm{for $j=1,\ldots,J$} 
\end{array}\right.
\label{eqn332}\end{equation}
where $W_j$ are given by (\ref{eqnmain3}).

Now, we proceed as in the proof of Theorem \ref{theorem310}. The present proof is very similar however requires longer calculations, which involves multiple use of the triangle inequality. 

Testing the main equation of (\ref{eqn330}) by $y(x,s)\mathbf{1}_{0,t}(s)$ yields:
\begin{equation}\begin{array}{cl}
			&\int_0^t \left<y',y\right> + D\xnorm{\nabla y}{2}^2 \:ds = \int_0^t (f(y^1)-f(y^2),y^1-y^2)_{\Lp{2}{\Omega}} + \\
			&\qquad\qquad + \sum_{j=1}^J(\hat{u}_{g_j}^1\kappa_j^1 - \hat{u}_{g_j}^2\kappa_j^2,y^1-y^2)_{\Lp{2}{\Omega}} \:ds\\
\end{array}\label{eqn334}\end{equation}
By the Lipschitz continuity of $f$ we have:
\begin{equation}\begin{array}{cl}
			& (f(y^1)-f(y^2),y^1-y^2)_{\Lp{2}{\Omega}} \leq L\xnorm{y^1-y^2}{2}
\end{array}\label{eqn336}\end{equation}
while for the second term on the right hand side of (\ref{eqn334}) we can write
\begin{equation}\begin{array}{cl}
			& (\hat{u}_{g_j}^1\kappa_j^1 - \hat{u}_{g_j}^2\kappa_j^2,y^1-y^2)_{\Lp{2}{\Omega}} = (\hat{u}_{g_j}^1\kappa_j^1 - \hat{u}_{g_j}^2\kappa_j^1,y)_{\Lp{2}{\Omega}} + (\hat{u}_{g_j}^2\kappa_j^1 - \hat{u}_{g_j}^2\kappa_j^2,y)_{\Lp{2}{\Omega}}\\
\leq	&\hnorm{\kappa_j^1}\xnorm{\hat{u}_{g_j}}{2}\xnorm{y}{2} + \hnorm{\kappa_j}\xnorm{\hat{u}_{g_j}^2}{2}\xnorm{y}{2}\\
\leq	&\frac{1}{2}\hnorm{\kappa_j^1}^2\xnorm{\hat{u}_{g_j}}{2}^2 + \frac{1}{2}\xnorm{y}{2}^2 + \frac{1}{2}\xnorm{\hat{u}_{g_j}^2}{2}^2\hnorm{\kappa_j}^2 + \frac{1}{2}\xnorm{y}{2}^2\\
\leq	&\frac{1}{2} C_1 \xnorm{\hat{u}_{g_j}}{2}^2 + \frac{1}{2} {(R^U)}^2 \hnorm{\kappa_j}^2 + \xnorm{y}{2}^2
\end{array}\label{eqn338}\end{equation}
where $C_1$ denotes the constant from the assertion of Theorem \ref{theorem310} --- it states that the square of the supremum of each $\kappa_j$ is bounded by this constant. Now, (\ref{eqn334}), (\ref{eqn336}) and (\ref{eqn338}) together imply
\begin{equation}\begin{array}{cl}
			& \int_0^t \left<y',y\right> + D\xnorm{\nabla y}{2}^2 \:ds \leq \int_0^t (L+J)\xnorm{y}{2}^2 + \sum_{j=1}^J\frac{(R^U)^2}{2}\hnorm{\kappa_j}^2 + \sum_{j=1}^J\frac{C_1}{2}\xnorm{\hat{u}_{g_j}}{2}^2 \:ds
\end{array}\label{eqn340}\end{equation}

A similar procedure can be performed for the equation for $\kappa_j$ --- we test the equations of (\ref{eqn332}) by $\kappa_j(s)\mathbf{1}_{0,t}(s)$ and neglect the $\hnorm{\kappa_j}^2$ terms:
\begin{equation}\begin{array}{cl}
			& \beta_j\int_0^t \kappa_j '\kappa_j\:ds \\
\leq 	&\int_0^t \sum_{k=1}^K \hnorm{\hat{u}_{\alpha_{jk}}^1w_k(\int_\Omega \hat{u}_{h_k}^1(y^1-y^*)\:dx) - \hat{u}_{\alpha_{jk}}^2w_k(\int_\Omega \hat{u}_{h_k}^2(y^2-y^*)\:dx)}\hnorm{\kappa_j}\:ds \\
\leq	&\int_0^t \frac{K}{2}\hnorm{\kappa_j}^2 + \sum_{k=1}^K \frac{1}{2}\hnorm{\hat{u}_{\alpha_{jk}}^1 w_k(\int_\Omega \hat{u}_{h_k}^1(y^1-y^*)\:dx) - \hat{u}_{\alpha_{jk}}^2 w_k(\int_\Omega \hat{u}_{h_k}^2(y^2-y^*)\:dx)}^2\:ds\\
\leq	&\int_0^t \frac{K}{2}\hnorm{\kappa_j}^2 + \sum_{k=1}^K \frac{1}{2}\hnorm{\hat{u}_{\alpha_{jk}}^1 w_k(\int_\Omega \hat{u}_{h_k}^1(y^1-y^*)\:dx) - \hat{u}_{\alpha_{jk}}^1 w_k(\int_\Omega \hat{u}_{h_k}^1(y^2-y^*)\:dx)}^2\\
+			&\sum_{k=1}^K \frac{1}{2}\hnorm{\hat{u}_{\alpha_{jk}}^1 w_k(\int_\Omega \hat{u}_{h_k}^1(y^2-y^*)\:dx) - \hat{u}_{\alpha_{jk}}^1 w_k(\int_\Omega \hat{u}_{h_k}^2(y^2-y^*)\:dx)}^2 \\
+			&\sum_{k=1}^K \frac{1}{2}\hnorm{\hat{u}_{\alpha_{jk}}^1 w_k(\int_\Omega \hat{u}_{h_k}^2(y^2-y^*)\:dx) - \hat{u}_{\alpha_{jk}}^2 w_k(\int_\Omega \hat{u}_{h_k}^2(y^2-y^*)\:dx)}^2\:ds\\
\end{array}\label{eqn342}\end{equation}
We estimate separately the terms appearing under the sums above. By the Lipschitz continuity of $w_k$ we get:
\begin{equation}\begin{array}{cl}
			& \hnorm{\hat{u}_{\alpha_{jk}}^1 w_k(\int_\Omega \hat{u}_{h_k}^1(y^1-y^*)\:dx) - \hat{u}_{\alpha_{jk}}^1 w_k(\int_\Omega \hat{u}_{h_k}^1(y^2-y^*)\:dx)}\\
\leq	& L_k\hnorm{\hat{u}_{\alpha_{jk}}^1}\xnorm{\hat{u}_{h_k}^1}{2}\xnorm{y^1 - y^2}{2} \leq L_k R^U R^U\xnorm{y^1 - y^2}{2}
\end{array}\label{eqn344}\end{equation}
The second term is estimated as follows:
\begin{equation}\begin{array}{cl}
			& \hnorm{\hat{u}_{\alpha_{jk}}^1 w_k(\int_\Omega \hat{u}_{h_k}^1(y^2-y^*)\:dx) - \hat{u}_{\alpha_{jk}}^1 w_k(\int_\Omega \hat{u}_{h_k}^2(y^2-y^*)\:dx)}\\
\leq	&  L_k\hnorm{\hat{u}_{\alpha_{jk}}^1}\xnorm{y^2 - y^*}{2}\xnorm{\hat{u}_{h_k}^1 - \hat{u}_{h_k}^1}{2} \leq  L_k R^U (C_1^{1/2} + \xnorm{y^*}{2})\xnorm{\hat{u}_{h_k}^1 - \hat{u}_{h_k}^1}{2}
\end{array}\label{eqn346}\end{equation}
 The last term in the right hand side of (\ref{eqn342}) obeys:
\begin{equation}\begin{array}{cl}
			& \hnorm{\hat{u}_{\alpha_{jk}}^1 w_k(\int_\Omega \hat{u}_{h_k}^2(y^2-y^*)\:dx) - \hat{u}_{\alpha_{jk}}^2 w_k(\int_\Omega \hat{u}_{h_k}^2(y^2-y^*)\:dx)}\\
\leq	& \hnorm{\hat{u}_{\alpha_{jk}}^1 - \hat{u}_{\alpha_{jk}}^2}\hnorm{w_k(\int_\Omega \hat{u}_{h_k}^2(y^2-y^*)\:dx)}\\
\leq	& \hnorm{\hat{u}_{\alpha_{jk}}^1 - \hat{u}_{\alpha_{jk}}^2}\left( w_{k0} + L_k \xnorm{\hat{u}_{h_k}^2}{2}\xnorm{y^2 - y^*}{2} \right)\\
\leq	& \hnorm{\hat{u}_{\alpha_{jk}}^1 - \hat{u}_{\alpha_{jk}}^2}\left( w_{k0} + L_k R^U\left(C_1^{1/2} + \xnorm{y^*}{2}\right) \right)\\
\end{array}\label{eqn348}\end{equation}
From (\ref{eqn342}), (\ref{eqn344}), (\ref{eqn346}) and (\ref{eqn348}) we infer that:
\begin{equation}\begin{array}{cl}
		& \int_0^t \kappa_j '\kappa_j\:ds\ \leq\ \frac{K}{2\beta_j}\int_0^t\hnorm{\kappa_j}^2 \:ds\ +\  \left( \frac{(R^U)^4}{2\beta_j}\sum_{k=1}^K L_k^2 \right)\int_0^t\xnorm{y}{2}^2 \:ds\ + \\
+		& \beta_j^{-1}\sum_{k=1}^K \xnorm{\hat{u}_{h_k}}{2}^2 \int_0^T L_k^2 (R^U)^2 \left(C_1 + \xnorm{y^*}{2}^2\right) \:ds\  + \\
+		& \beta_j^{-1}\sum_{k=1}^K \hnorm{\hat{u}_{\alpha_{jk}}}^2 \int_0^T \left( w_{k0}^2 + L_k^2(R^U)^2(C_1 + \xnorm{y^*}{2}^2) \right)\:ds \\
\leq	& \frac{K}{2\beta_j}\int_0^t\hnorm{\kappa_j}^2 \:ds\ +\ C_{2,j} \int_0^t\xnorm{y}{2}^2 \:ds\ +\ C_{3,j} \sum_{k=1}^K \xnorm{\hat{u}_{h_k}}{2}^2\ +\ C_{4,j} \sum_{k=1}^K \hnorm{\hat{u}_{\alpha_{jk}}}^2 
\end{array}\label{eqn350}\end{equation}
where
\begin{displaymath}\begin{array}{rcl}
C_{2,j} & = & C_{2,j}(R^U,L_1,\ldots,L_K,\beta_j) \\
C_{3,j} & = & C_{3,j}(R^U,L_1,\ldots,L_K,\beta_j,\xnorm{y^*}{2,2},C_1) \\
C_{4,j} & = & C_{4,j}(R^U,L_1,\ldots,L_K,\beta_j,w_{10},\ldots,w_{K0},\xnorm{y^*}{2,2},C_1) \\
\end{array}\end{displaymath}

We sum (\ref{eqn340}) and (\ref{eqn350}) for every $j$, use the identities $2\int_0^t\left<y',y\right> = \xnorm{y(\,.\,,t)}{2}^2 - \xnorm{y(\,.\,,0)}{2}^2$ and $2\int_0^t\kappa_j '\kappa_j\ = \hnorm{\kappa_j(t)}^2 - \hnorm{\kappa_j(0)}^2$ and neglect the gradient term (compare with (\ref{eqn314}) in the proof of Theorem \ref{theorem310}). As the result, we get: 
\begin{equation}\begin{array}{cl}
			&  \xnorm{y(\,.\,,t)}{2}^2 + \sum_{j=1}^J\hnorm{\kappa_j(t)}^2 \ \leq\ \xnorm{y_0}{2}^2 + \sum_{j=1}^J\hnorm{\kappa_{j0}}^2\ +\\
+			& 2(L+J+\sum_{j=1}^JC_{2,j})\int_0^t\xnorm{y}{2}^2 \:ds\ +\ ((R^U)^2+\max\{K\beta_j^{-1}\})\sum_{j=1}^J\int_0^t\hnorm{\kappa_j}^2 \:ds\ +\\
+			& C_1 \sum_{j=1}^J\xnorm{\hat{u}_{g_j}}{2}^2\ +\ 2\sum_{j=1}^J C_{3,j}\sum_{k=1}^K\xnorm{\hat{u}_{h_k}}{2}^2\ +\ 2\sum_{j=1}^J \sum_{k=1}^K C_{4,j} \hnorm{\hat{u}_{\alpha_{jk}}}^2\\
\leq	& C_5 \int_0^t\xnorm{y}{2}^2 + \sum_{j=1}^J \hnorm{\kappa_j}^2 \:ds + C_6 \xnorm{\hat{u}}{U}^2
\end{array}\label{eqn352}\end{equation}
where
\begin{displaymath}\begin{array}{rcl}
C_5 & = & \max\left\{ 2(L+J+JC_2), (R^U)^2+ \max_{j=1,\ldots,J}\frac{K}{\beta_j} \right\} \\
C_6 & = & \max\left\{ C_1, 2\sum_{j=1}^J C_{3,j}, 2C_{4,j}\right\} \\
\end{array}\end{displaymath}
By the integral Gr\"onwall inequality we infer from (\ref{eqn352}) that 
\begin{equation}\begin{array}{cl}
			&\xnorm{y}{2,\infty}^2 + \sum_{j=1}^J \xnorm{\kappa_j}{\Lp{\infty}{0,T}}^2\ \leq\\ 
\leq	&(1 + T C_5e^{C_5})\left(\xnorm{y_0}{2}^2 + \left(\sum_{j=1}^J \hnorm{\kappa_{j0}}^2\right) + C_6 \xnorm{\hat{u}}{U}^2\right)
\end{array}\label{eqn354}\end{equation}
To close the proof, it suffices to show that 
\begin{equation}\begin{array}{cl}
			&\xnorm{\nabla y}{2,2}\ +\ \xnorm{y'}{\dual{\Ha{1}},2}\ +\ \xnorm{\kappa_j'}{\Lp{2}{0,T}}\\
\leq	& C\left( \xnorm{y}{2,\infty}^2\ +\ \sum_{j=1}^J\xnorm{\kappa_j}{\Lp{\infty}{0,T}}^2 \right) \\
\end{array}\label{eqn356}\end{equation}
where $C$ depends only on $D$ and on the same quantities as $C_5$, $C_6$ and then (\ref{eqn354}) can be applied to complete our reasoning. But the estimates for particular norms appearing in (\ref{eqn356}) can be obtained with the same methods as in the proof of Theorem \ref{theorem310} (see (\ref{eqn318}), (\ref{eqn320}) and (\ref{eqn322})).
\end{proof}
In consequence of Theorems \ref{theorem04.0} and \ref{theorem310}, we obtain the following:
\begin{coro}
Assume that general assumptions \asmstylei{asm02} - \asmstylei{asm12} hold and $\left(g_j,h_k,\alpha_{jk} \right)_{j=1,\ldots,J}^{k=1,\ldots,K}\in U$. Then system (\ref{eqnmain1}) - (\ref{eqnmain3}) has a weak solution.
\label{theorem04.1}\end{coro}
\begin{proof}
Theorem \ref{theorem04.1} assumes that $w_k$ functions are bounded, i.e. $\xnorm{w_k}{\Lp{\infty}{\R}}<\infty$. But Theorem \ref{theorem310} gives a bound for solutions of (\ref{eqnmain1}) - (\ref{eqnmain3}) that is independent on $\xnorm{w_k}{\Lp{\infty}{\R}}$. Thus the standard truncation technique can be utilized to dispense the assumption that $w_k$ are bounded. 

Namely, by Theorem \ref{theorem310} $\xnorm{y}{2,\infty}=C_y<\infty$ where $C_y$ doesn't depend on $\xnorm{w_k}{\Lp{\infty}{\R}}$. For given $w_k$ as in assumption \asmstylei{asm06}, consider its truncation $w_k^n$ given by $w_k^n(s):=w_k(s)$ for $s$ such that $\hnorm{w_k(s)}\leq n$, $w_k^n(s) = n$ for $s$ s. th. $w_k(s)> n$ and $w_k^n(s) = -n$ for $s$ s. th. $w_k(s)< -n$. Observe that $n=\xnorm{w_k^n}{\Lp{\infty}{0,T}}$. Since $C_y$ is independent on $\xnorm{w_k^n}{\Lp{\infty}{0,T}}$, there exists $n$ such that $n>\xnorm{h_k}{2}\left(C_y+\xnorm{y^*}{2,\infty}\right)$. Taking $n$ satisfying this condition, the switching function $w_k$ in (\ref{eqnmain3}) can be replaced by $w_k^n$ with no side effect to the weak solution of (\ref{eqnmain1}) - (\ref{eqnmain3}) because $\int_\Omega h_k(y(\,.\,,t)-y^*(\,.\,,t))\:dx\leq \xnorm{h_k}{2}\left(C_y+\xnorm{y^*}{2,\infty}\right)$ for a.e. $t\in [0,T]$. For such $w_k^n$, Theorem \ref{theorem04.0} can be applied.
\end{proof}
The below corollary is straightforward due to Corollary \ref{theorem04.1} and Theorem \ref{theorem313}.
\begin{coro}
Let the general assumptions \asmstylei{asm02} - \asmstylei{asm12} be fulfilled and $\left(g_j,h_k,\alpha_{jk} \right)_{j=1,\ldots,J}^{k=1,\ldots,K}\in U$. Then system (\ref{eqnmain1}) - (\ref{eqnmain3}) has a unique weak solution.
\label{theorem04.2}\end{coro}
This closes the part concerning the uniqueness and existence of the weak solutions of (\ref{eqnmain1}) - (\ref{eqnmain3}). 
However, Theorem \ref{theorem310} and Theorem \ref{theorem313} are necessary not only for the uniqueness and existence theorem in Corollary \ref{theorem04.2}. The stability result in Theorem \ref{theorem313} will be utilized in the part of our research that concerns optimal location of the devices (see introduction).

But there are also other properties concerning the behavior of system (\ref{eqnmain1}) - (\ref{eqnmain3}) under the perturbations of the control which we would like to present. 
Assume that there is a sequence of controls $\hat{u}^n\in U$ given and one have only the knowledge on the weak convergence of this controls. This doesn't allow to utilize the former theorems of the present section to infer about anything more than boundedness of $(y^n,\kappa_1^n,\ldots,\kappa_J^n)$ in $X^2$, where $(y^n,\kappa_1^n,\ldots,\kappa_J^n)$ denotes the solution of (\ref{eqnmain1}) - (\ref{eqnmain3}) corresponding to $\hat{u}^n$.
Hence the following theorem:
\begin{theorem}
Let the assumptions \asmstylei{asm02} - \asmstylei{asm12} be fulfilled. Let the sequence $\hat{u}^n$ converge weakly to $\hat{u}$ in $U$. Denote by $(y^n,\kappa_1^n,\ldots,\kappa_J^n)$ the solution of (\ref{eqnmain1}) - (\ref{eqnmain3}) corresponding to $\hat{u}^n$ and by $(\widetilde{y},\widetilde{\kappa}_1,\ldots,\widetilde{\kappa}_J)$ the solution of (\ref{eqnmain1}) - (\ref{eqnmain3}) corresponding to $\hat{u}$. Then there exists a sequence of natural indexes $n_1<n_2<\ldots$ such that subsequence  $(y^{n_k},\kappa_1^{n_k},\ldots,\kappa_J^{n_k})$ converges weakly-$*$ to $(\widetilde{y},\widetilde{\kappa}_1,\ldots,\widetilde{\kappa}_J)$ in $X^2$ when $k\rightarrow\infty$. 
\label{theorem314}\end{theorem}
\begin{proof}
Let $\hat{u}^n\rightharpoonup\hat{u}$ in $U$, as in the assumptions. A weakly convergent sequence is bounded, thus by Theorem \ref{theorem310} $(y^n,\kappa_1^n,\ldots,\kappa_J^n)$ is bounded in $X^2$. This allows us to extract a weakly-$*$ convergent subsequence (for simplicity, we relabel it and keep the original indexes): $(y^n,\kappa_1^n,\ldots,\kappa_J^n)\stackrel{\ast}{\rightharpoonup}(\bar{y},\bar{\kappa}_1,\ldots,\bar{\kappa}_J)$ in $X^2$ for certain $(\bar{y},\bar{\kappa}_1,\ldots,\bar{\kappa}_J)\in X^2$. In particular:
\begin{equation}\begin{array}{ll}
y^n\stackrel{\ast}{\rightharpoonup}\bar{y} & \textrm{in $\eLp{\infty}{\Lp{2}{\Omega}}$}\\
{y^n}'\rightharpoonup\bar{y}' & \textrm{in $\eLp{2}{\dual{\Ha{1}}}$}\\
\nabla y^n\rightharpoonup\nabla\bar{y} & \textrm{in $\left(\Lp{2}{Q_T}\right)^d$}\\
\kappa_j^n\stackrel{\ast}{\rightharpoonup}\bar{\kappa}_j &\textrm{in $\Lp{\infty}{0,T}$}\\
{\kappa_j^n}'\rightharpoonup{\bar{\kappa}_j}' &\textrm{in $\Lp{2}{0,T}$}\\
\end{array}
\label{eqn360}\end{equation}

It suffices to show that $(\bar{y},\bar{\kappa}_1,\ldots,\bar{\kappa}_J)=(\widetilde{y},\widetilde{\kappa}_1,\ldots,\widetilde{\kappa}_J)$. For this reason we need to prove that we can pass with $n$ to infinity in all terms appearing in the weak formulation given in Definition \ref{def04} 
The passage in linear terms follows straight due to (\ref{eqn360}). We are left to deal with the terms
\begin{displaymath}
\int_0^T (\kappa_j^n \hat{u}_{g_j}^n,\phi)_{\Lp{2}{\Omega}}\:dt,\quad \int_0^T (f(y^n),\phi)_{\Lp{2}{\Omega}}\:dt,\quad \int_0^T W_j(y^n,y^*)\xi\:dt
\end{displaymath}
for $\phi\in\eLp{2}{\Ha{1}}$, $\xi\in\Lp{2}{0,T}$.

Let us begin with the term corresponding to $\kappa_j^n \hat{u}_{g_j}^n$. By the assumption and by (\ref{eqn360}), $\hat{u}_{g_j}\rightharpoonup\hat{u}$ in $\Lp{2}{\Omega}$ and $\kappa_j^n\rightharpoonup\bar{\kappa_j}$ in $\Lp{2}{0,T}$. But this means that for arbitrary $\phi^\Omega\in \Cont{\bar{\Omega}}$ and $\phi^T\in \Cont{[0,T]}$ we have
\begin{displaymath}\begin{array}{cl}
								&\int_0^T (\kappa_j^n \hat{u}_{g_j}^n,\phi^\Omega\phi^T)_{\Lp{2}{\Omega}} \:dt = \int_0^T \kappa_j^n \phi^T \:dt \int_\Omega \hat{u}_{g_j}^n  \phi^\Omega\:dx \longrightarrow\\
\longrightarrow	& \int_0^T \bar{\kappa}_j \phi^T \:dt \int_\Omega \hat{u}_{g_j}  \phi^\Omega\:dx = \int_0^T (\bar{\kappa}_j \hat{u}_{g_j},\phi^\Omega\phi^T)_{\Lp{2}{\Omega}} \:dt 
\end{array}\end{displaymath}
To conclude that the weak convergence of $\kappa_j^n \hat{u}_{g_j}^n$ to $\bar{\kappa}_j \hat{u}_{g_j}$ in $\Lp{2}{Q_T}$ holds it suffices to justify that  $\kappa_j^n \hat{u}_{g_j}^n$ is bounded in $\Lp{2}{Q_T}$ and the set of functions $\phi$ of form $\phi(x,t)=\phi^\Omega(x)\phi^T(t)$, where $\phi^\Omega$ and $\phi^T$ are as above, is linearly dense in $\Lp{2}{Q_T}$. The former is straightforward by the weak convergence properties of $\kappa_j^n$ and $\hat{u}_{g_j}^n$. Concerning the latter, by the Stone-Weierstrass theorem (see \cite[Chap. 0.2, p.9]{Yosida66}), the set of all possible $\phi$ is dense in $C(\bar{Q}_T)$ and the latter set is linearly dense in $\Lp{2}{Q_T}$. Altogether, the following can be stated:
\begin{equation}\begin{array}{ll}
\kappa_j^n \hat{u}_{g_j}^n\rightharpoonup \bar{\kappa}_j \hat{u}_{g_j} & \textrm{in $\Lp{2}{Q_T}$}\\
\end{array}\label{eqn362}\end{equation}

Guaranteeing the convergence of the remaining two terms will involve the knowledge on the strong convergence of $y^n$ in $\Lp{2}{Q_T}$. But by the bounds for $y^n$ and ${y^n}'$ in (\ref{eqn360}) and by the Aubin-Lions lemma (see \cite[Chap III.1. Prop. 1.3]{sho} for the probably most common formulation of the lemma or \cite[Sec. 8 Cor. 4]{sim} for a more general statement), there exists a subsequence such that
\begin{equation}\begin{array}{ll}
y^n\rightarrow\bar{y} & \textrm{in $\Lp{2}{Q_T}$}\\
\end{array}
\label{eqn364}\end{equation}
The limit in (\ref{eqn364}) is exactly $\bar{y}$ since otherwise it would be a contradiction to (\ref{eqn360}).

By the Lipschitz continuity of $f$ and (\ref{eqn364}) the convergence
\begin{equation}\begin{array}{ll}
f(y^n)\rightarrow f(\bar{y}) & \textrm{in $\Lp{2}{Q_T}$}\\
\end{array}
\label{eqn366}\end{equation}
is a straightforward conclusion.

We are left to investigate the convergence of the term corresponding to $W_j(y^n,y^*)$. Note that by the definition (see (\ref{eqnmain3})), $W_j$ has an implicit dependence on $\hat{u}_{h_k}^n$ and $\hat{u}_{\alpha_{jk}}^n$. Thus in the present context we should interpret $W_j$ as $W_j(\hat{u}_{h_k}^n,\hat{u}_{\alpha_{jk}}^n,y^n,y^*)$. By (\ref{eqnmain3}) and the Lipschitz continuity of $w_k$ we can write, using the triangle inequality:

\begin{equation}
\begin{array}{cl}
	&\int_0^T \hnorm{W_j(\hat{u}_{h_k}^n,\hat{u}_{\alpha_{jk}}^n,y^n(t),y^*(t)) - W_j(\hat{u}_{h_k},\hat{u}_{\alpha_{jk}},\bar{y}(t),y^*(t))}^2\:dt\\
\leq	&\begin{array}[t]{lcl}
	2\sum_{k=1}^K L_k \big\{ &\quad &\hnorm{\hat{u}_{\alpha_{jk}}^n}^2\xnorm{\hat{u}_{h_k}^n}{2}^2\int_0^T \xnorm{y^n - \bar{y}}{2}^2\:dt\\
	    &+& \hnorm{\hat{u}_{\alpha_{jk}}^n}^2\int_0^T\hnorm{\int_\Omega (\hat{u}_{h_k}^n - \hat{u}_{h_k})(\bar{y}-y^*)\:dx}^2\:dt\\
	    &+& \hnorm{\hat{u}_{\alpha_{jk}}^n - \hat{u}_{\alpha_{jk}}}^2\xnorm{\hat{u}_{h_k}}{2}^2\int_0^T \xnorm{\bar{y}-y^*}{2}^2\:dt \quad \big\}\\
	\end{array}\\
\end{array}
\label{eqn368}\end{equation}
Let us consider each of the three terms appearing in the right hand side of the above.

The first term in the right hand side of (\ref{eqn368}) converges to zero since the sequence of controls $\hat{u}^n$ is bounded and (\ref{eqn364}) holds. 

The third term in the right hand side of (\ref{eqn368}) is convergent to zero since by $\hat{u}^n\rightharpoonup\hat{u}$ in $U$ we have $\hat{u}_{\alpha_{jk}}^n\rightarrow \hat{u}_{\alpha_{jk}}$.

To treat the second term, consider a function
\begin{displaymath}
F^n(t) = \int_\Omega (\hat{u}_{h_k}^n - \hat{u}_{h_k})(\bar{y}(t)-y^*(t))\:dx
\end{displaymath}
As the sequence of numbers $\hnorm{\hat{u}_{\alpha_{j,k}}^n}^2$ in the considered term is bounded, it is enough to show the convergence of $F^n$ to zero in $\Lp{2}{0,T}$. We have $\bar{y}(t),y^*(t)\in\Lp{2}{\Omega}$ a.e. on $[0,T]$. Thus, by the weak convergence  $\hat{u}_{h_k}^n\rightharpoonup \hat{u}_{h_k}$ in $\Lp{2}{\Omega}$ for every $k=1,\ldots,K$ we infer that $F^n(t)$ converges to zero a.e. on $[0,T]$, as $n\rightarrow\infty$. Moreover, a.e. on $[0,T]$
\begin{displaymath}
\hnorm{F^n(t)}\ \leq\ \xnorm{\hat{u}_{h_k}^n - \hat{u}_{h_k}}{2}\xnorm{\bar{y}(t)-y^*(t)}{2}\ \leq\ C_U \xnorm{\bar{y}(t)-y^*(t)}{2}
\end{displaymath}
where $C_U = \sup_{n}\xnorm{\hat{u}^n}{U}$ is finite and the term $\xnorm{\bar{y}(t)-y^*(t)}{2}$ is square integrable due to $\bar{y},y^*\in\Lp{2}{Q_T}$. These observations concerning $F^n(t)$ allow us to apply the the Lebesgue dominated convergence theorem (see \cite[Chap. 1]{rud.realcomplex.en} or \cite[App. E.3, Th. 5]{ev}) and get the convergence
\begin{displaymath}
F^n \ \rightarrow 0 \quad \textrm{in $\Lp{2}{0,T}$}
\end{displaymath}

Altogether, we conclude that the right hand side of (\ref{eqn368}) converges to zero thus:
\begin{equation}\begin{array}{ll}
W_j(\hat{u}_{h_k}^n,\hat{u}_{\alpha_{jk}}^n,y^n,y^*)\rightarrow W_j(\hat{u}_{h_k},\hat{u}_{\alpha_{jk}},\bar{y},y^*) & \textrm{in $\Lp{2}{0,T}$}\\
\end{array}
\label{eqn370}\end{equation}

To sum up, the convergence results (\ref{eqn360}), (\ref{eqn362}), (\ref{eqn366}), (\ref{eqn370}) allow us to infer that $(\bar{y},\bar{\kappa}_1,\ldots,\bar{\kappa}_J)$ is the weak solution of system (\ref{eqnmain1}) - (\ref{eqnmain3}) in sense of the Definition \ref{def04}, corresponding to $\hat{u}$, i.e. $(\bar{y},\bar{\kappa}_1,\ldots,\bar{\kappa}_J)=(\widetilde{y},\widetilde{\kappa}_1,\ldots,\widetilde{\kappa}_J)$ what concludes the proof.
\end{proof}

\section{Numerical prototypes}\label{sec:CP.numexamples}\setcounter{equation}{0}

The present section is devoted to numerical simulations for illustrating some of the properties of the thermostats control system utilized in (\ref{eqnmain1}) - (\ref{eqnmain3}). Before we pass to presentation of the numerical results (Section \ref{sec:CP.numexamples.results}), we describe precisely the configuration of system (\ref{eqnmain1}) - (\ref{eqnmain3}) which was utilized in the experiments (Section \ref{sec:CP.numexamples.assumptions.structural}) and give a description of utilized numerical methods (Section \ref{sec:CP.numexamples.assumptions.numerics}).

In Section \ref{sec:CP.numexamples.results}, we consider three situations. The first concerns the behavior of the \tcb{thermostats control system} when it is focused on a task of preserving an unstable state (Section \ref{sec:CP.numexamples.results:exp_unstable}). The second one concerns an attempt of comparison of efficiency of the \tcb{thermostats control system} for various initial conditions (Section \ref{sec:CP.numexamples.results:exp_icond}). 
The third one compares the behavior of the \tcb{thermostats control system} when two different choices of the control and measurement devices are considered (Section \ref{sec:CP.numexamples.results:exp_linefb}).

If the behavior of the process stabilizes for large times, it is meaningful to measure the quality of a particular control $\hat{u}\in U$ as some functional of the process state in the final time of the experiment. The stabilization occurs to be the case in the experiments described in Section \ref{sec:CP.numexamples.results}. Hence, we emphasize the results concerning the \tcb{thermostats control system} efficiency understood in the above manner.

\subsection{Structural assumptions}\label{sec:CP.numexamples.assumptions.structural}

In the experiments described in Section \ref{sec:CP.numexamples.results} the below assumptions were made.

We assumed, that every control device in system (\ref{eqnmain1}) - (\ref{eqnmain3}) is described by a characteristic function of a disc centered at $x_{j}\in\Omega$ of the same radius, times a constant. We also put $K=J$ and made an analogous assumption for the measurement devices. More precisely, the devices were determined by 
\begin{equation}
g_j := \hat{u}_{g_j}:=\sigma_g(\,.\, - x_{j})|_\Omega \qquad h_j(x) := \hat{u}_{h_j} := \sigma_h(\,.\, - x_{j})|_\Omega
\label{asm:01.2}\end{equation}
for $x_j\in\Omega$, $j=1,\ldots,J$ and for:
\begin{equation}
\sigma_g(x) = C_g\mathbf{1}_{B(0,r_\sigma)}(x) \qquad \sigma_h(x) = C_h\mathbf{1}_{B(0,r_\sigma)}(x)
\label{asm:01.4}\end{equation} 
for certain $r_\sigma,C_g,C_h>0$. In other words, every control device was covered by exactly one measurement device, however they could differ in the height parameter.

In this situation, the following assumptions concerning the weights were natural: we set $\alpha_{jk} := \hat{u}_{\alpha_{jk}} := \delta_{j,k}$, where $\delta_{j,k}$ denotes the Kronecker delta function of $j$ and $k$.

Having the above, the control $\hat{u} := \left({g_j},{h_j},{\alpha_{jk}}\right)_{j=1,\ldots,J}\in U$ is determined once a selection of the points $x_{1},\ldots,x_{J}$ and the parameters $r_\sigma,C_g,C_h > 0$ is made.

The above assumptions result in a simplified version of the model (\ref{eqnmain1}) - (\ref{eqnmain3}), which is a focus of our interest in the present Section, concerning the numerical results:
\begin{equation}
\left\{\begin{array}{ll}
y_t(x,t) - D\Delta y(x,t) = f(y(x,t)) + \sum_{j=1}^J g_j(x)\kappa_j(t) & \textrm{on $Q_T$}\\
\frac{\partial y}{\partial n} = 0 & \textrm{on $\partial\Omega\times (0,T)$}\\
y(0,x)=y_0(x) & \textrm{for $x\in\Omega$}
\end{array}\right.
\label{eqnmain1num}\end{equation}
together with
\begin{equation}
\left\{\begin{array}{ll}
\beta_1\kappa_1'(t) + \kappa_1(t) = w_1\left( \int_\Omega h_1(x) (y - y^*) dx \right)& \textrm{on $[0,T]$}\\
\vdots &\vdots\\
\beta_J\kappa_J'(t) + \kappa_J(t) = w_J\left( \int_\Omega h_J(x) (y - y^*) dx \right)& \textrm{on $[0,T]$}\\
\kappa_j(0)=\kappa_{j0}\in\R & \textrm{for $j=1,\ldots,J$} 
\end{array}\right.
\label{eqnmain2num}\end{equation}
for $g_j$ and $h_j$ functions defined by (\ref{asm:01.2}) and (\ref{asm:01.4}).

The experiments were performed for a two-dimensional rectangular domain: 
\begin{equation}
\Omega=(-1,1)\times (-1,1)\subset\R^2
\label{asm:02}\end{equation}

It was assumed that $y^*$ was time independent: $y^* = y^*(x)$.

The reactive term $f$ treated in the experiments was:
\begin{equation}
 f(s) = -s^3 + s
\label{asm:03.f}\end{equation}
together with $w_j$, $j=1,\ldots,J$ given by
\begin{equation}
w_j(s) = H_w \max(\min(L_w s, 1),-1)
\label{asm:03.w}\end{equation}
for certain $L_w$, $H_w$.

For a given $r_\sigma$, we considered the value of $C_h$ to be determined by the following relation:
\begin{equation}
C_{switch} \int_{\R^d} \sigma_h = 1/\hnorm{L_w}
\label{asm:06}\end{equation}
for certain $C_{switch}>0$. In the above, $C_h$ is present in the definition of $\sigma_h$. Identity (\ref{asm:06}) along with definition of $\sigma_h$ in (\ref{asm:01.4}) allows to infer that
\begin{equation}
C_h = \left( \pi\,\hnorm{L_w}\,C_{switch}\, r_\sigma^2 \right)^{-1}
\label{eqn:010}\end{equation}

For better explanation of the meaning of the constant $C_{switch}>0$, we make the following remark. Let us temporary call the term $w_j\left( \int_\Omega h_j (y - y^*) \right)$ in the right hand side of \ref{eqnmain2num} \emph{a signal demand} of $j$-th measurement device. The concept is that $C_{switch}$ defines a threshold deviation between the solution $y$ and the reference state $y^*$ after exceeding which the extremal value of signal demand is returned to $j$-th signal generator by $j$-th measurement device. Being more precise, for a given measurement device described by a function $h_j$ we want the signal demand to achieve its maximal value when $y-y^* \approx C_{switch}$ or $y -y^*\approx -C_{switch}$, at least in the support of $h_j$. But taking the formula for $w_j$ into account, its extremal value is achieved if $\int_\Omega h_j (y - y^*) \approx\pm 1/\hnorm{L_w}$. Processing the above conditions yields
\begin{displaymath}
1/\hnorm{L_w} \ \approx\ \int_\Omega h_j \hnorm{y - y^*} \ \approx\  C_{switch} \int_\Omega h_j 
\end{displaymath}
This gives relation (\ref{asm:06}) after assuming that $\approx$ sign can be replaced by the equality sign and assuming that $\int_\Omega h_j = \int_{\R^d} \sigma_h$. The latter is correct if $\textrm{supp}(\sigma_h(\,.\,-x_j))\subset\Omega$. We considered it as a usual situation in order to give clear interpretation of constant $C_{switch}$. However it can be not true in general.

Altogether, for $\Omega$ given by (\ref{asm:02}), the reactive term as in (\ref{asm:03.f}), the switching function $w_j$ as in (\ref{asm:03.w}), $g_j$, $h_j$ defined in the formula (\ref{asm:01.2}) with (\ref{asm:01.4}) and $C_h$ as in the formula (\ref{eqn:010}), system (\ref{eqnmain1num}) - (\ref{eqnmain2num}) is uniquely determined by the choice of the following quantities:
\begin{displaymath}
\begin{array}{ll}
y_0,\ \kappa_{10},\ldots,\kappa_{J0},\quad y^* 	& \quad J,\quad x_{1},\ldots,x_{J}\\
T,\quad D,\beta_1,\ldots,\beta_J, 			& \quad r_\sigma, C_g, C_{switch},L_w,H_w\\
\end{array}
\end{displaymath}
The values of the above quantities utilized in the particular experiments will be specified in Section \ref{sec:CP.numexamples.results}.

\subsection{Assumptions for the numerical scheme}\label{sec:CP.numexamples.assumptions.numerics}
The below numerical methods were utilized in the experiments described in Section \ref{sec:CP.numexamples.results}.

For numerical treatment of system (\ref{eqnmain1num}) - (\ref{eqnmain2num}) we utilized the finite element method to solve the component $y$ corresponding to the parabolic equation.

The triangulation of the $\Omega$, see (\ref{asm:02}), was of the type presented on Figure \ref{fig:triangulation}.
\begin{figure}[htb]
\begin{center}
\includegraphics[width = 0.245\textwidth]{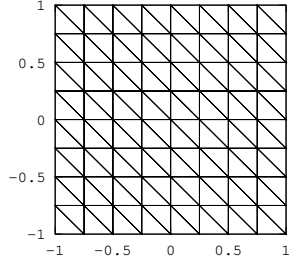}
\caption{The type of triangulation utilized in the experiment.\label{fig:triangulation}}
\end{center}
\end{figure}
The finite element space chosen for the simulations was the space of continuous functions, linear on every element of the triangulation. The implicit Euler scheme was used for the purpose of the time discretization of the model. The nonlinear terms $f$ and $w$ were treated with use of the Picard iterations technique. 

The constant amount of the Picard iterations for every time step was assumed instead of applying the error-based stop criterion in order to control the computational time.

In the further part of our work, we will use the following notation concerning the above described numerical scheme:\\
\begin{tabular}{lcl}
$N+1$				& --- & The number of vertices in the mesh along each spatial direction,\\
$h$				& --- & The length of the mesh step along each spatial direction,\\
$M+1$				& --- & The numer of points in the time discretization,\\
$\tau$		&	--- & The length of the time step,\\
$N_{Picard}$		&	--- & The number of Picard iterations in every time step\\
\end{tabular}\\
According to the above notation, the total number of vertices in the triangulation equals $(N+1)^2$. Moreover, the relations $h = N^{-1}$ and $\tau = M^{-1}$ hold. 

Note, that the numerical scheme that we apply is fully determined (up to the matrix assembly method) by the choice of the parameters determining the finite element space, the time discretization scheme and the nonlinear term treatment method, i.e. by the following parameters:
\begin{displaymath}
 N,\qquad M,\qquad N_{Picard}
\end{displaymath}
The values of the above parameters utilized in the particular experiments will be specified in section \ref{sec:CP.numexamples.results}.

\subsection{Results of the simulations}\label{sec:CP.numexamples.results}
\newcounter{experiment.CP}
\setcounter{experiment.CP}{1}

Now we proceed to presentation of the results announced in the introduction to Section \ref{sec:CP.numexamples}. The described below experiments were performed with use of the methods from Section \ref{sec:CP.numexamples.assumptions.numerics} and under structural assumptions from Section \ref{sec:CP.numexamples.assumptions.structural}.

In the presentation, some plots appear and thus we give a short clarification of the utilized plot convention here. The plots can be grouped into certain classes: the plots of the main component $y$ of the numerical solution of system (\ref{eqnmain1num}) - (\ref{eqnmain2num}), in a given moment of time, the plots of initial conditions or reference states, the plots concerning configuration of the control devices and the error plots. While the error plots are self-describing, the rest of the plots need to be commented.

In the plots of the main component $y$ of the numerical solution of system (\ref{eqnmain1num}) - (\ref{eqnmain2num}), in a given moment of time, the color map extends from black for certain down limit value $\texttt{min}$ to white for certain upper limit value \texttt{max}. The values exceeding the interval $[\texttt{min},\texttt{max}]$ are truncated to the values $\texttt{min}$ or $\texttt{max}$. The same convention holds for the plots of the $y_0$ components of the initial condition of system (\ref{eqnmain1num}) - (\ref{eqnmain2num}) and the plots of the reference state $y^*$. 

The default values of the color map parameters are $\texttt{min} = -1$ and $\texttt{max} = +1$. If the values $\texttt{min}$ or $\texttt{max}$ are different than the mentioned default ones, then it is indicated in the legend of the plot.

The plots concerning the configuration of the control and measurement devices are visualizations of supports of the functions $g_j$ and $h_j$ describing the control and measurement devices. This make sense since due to the structural assumptions in the Section \ref{sec:CP.numexamples.assumptions.structural}, the supports of the functions describing the control and measurement devices are pairwise equal. The difference between $g_j$ and $h_j$ for a given $j=1,\ldots,J$ can occur only in height parameters $C_g$ and $C_h$, appearing in (\ref{asm:01.4}). Thus, such visualizations, if sufficiently precise, give an unique characterization of the parameter $r_\sigma$ and of the utilized sequence of the central points, $x_1,\ldots,x_J$, appearing in in (\ref{asm:01.4}).

\begin{figure}[htb]
  \begin{center}
  \begin{subfigure}{0.245\textwidth}\begin{center}
  \includegraphics[width = \textwidth]{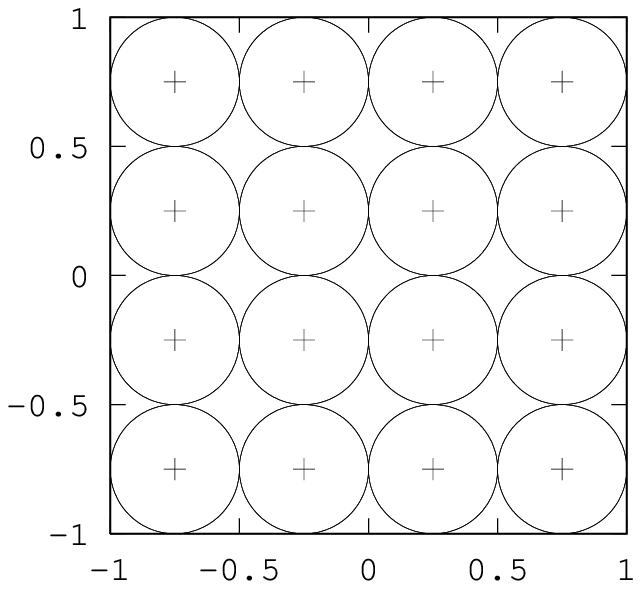}
  \caption{16 devices\label{fig:devices16_a}}
  \end{center}\end{subfigure}
  \begin{subfigure}{0.245\textwidth}\begin{center}
  \includegraphics[width = \textwidth]{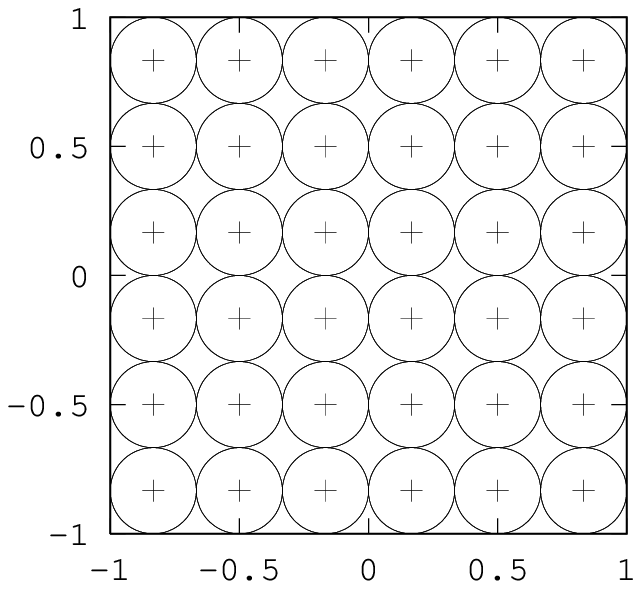}
  \caption{36 devices\label{fig:devices36_a}}
  \end{center}\end{subfigure}
  \begin{subfigure}{0.245\textwidth}\begin{center}
  \includegraphics[width = \textwidth]{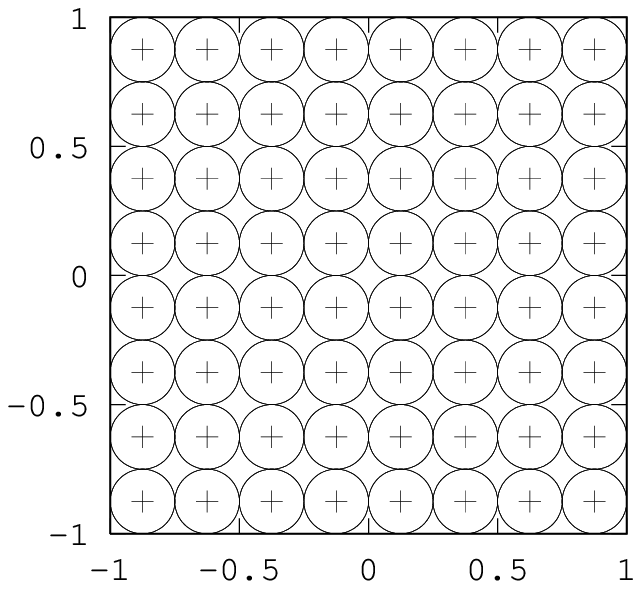}
  \caption{64 devices\label{fig:devices64_a}}
  \end{center}\end{subfigure}
  \begin{subfigure}{0.245\textwidth}\begin{center}
  \includegraphics[width = \textwidth]{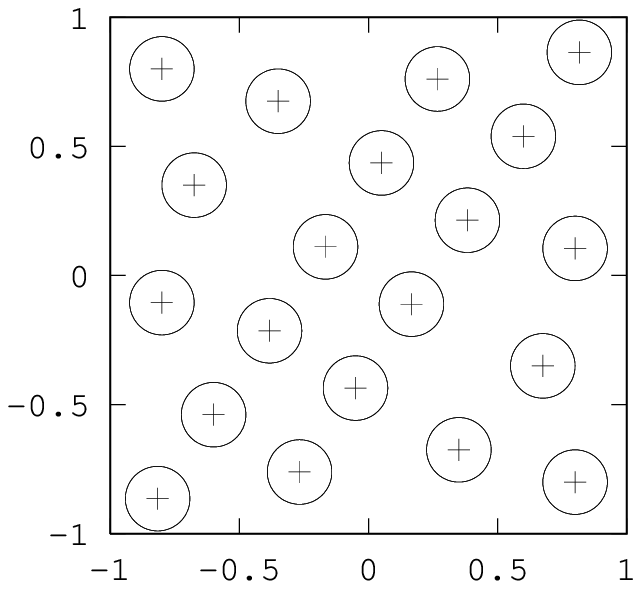}
  \caption{20 devices\label{fig:devices20_a}}
  \end{center}\end{subfigure}
  \caption{Control and measurement devices configurations for Section \ref{sec:CP.numexamples.results}.}\label{fig:sub:devconfdata}
  \end{center}
\end{figure}

\begin{figure}[htb]
  \begin{center}
  \begin{subfigure}{0.245\textwidth}\begin{center}
  \includegraphics[width = \textwidth]{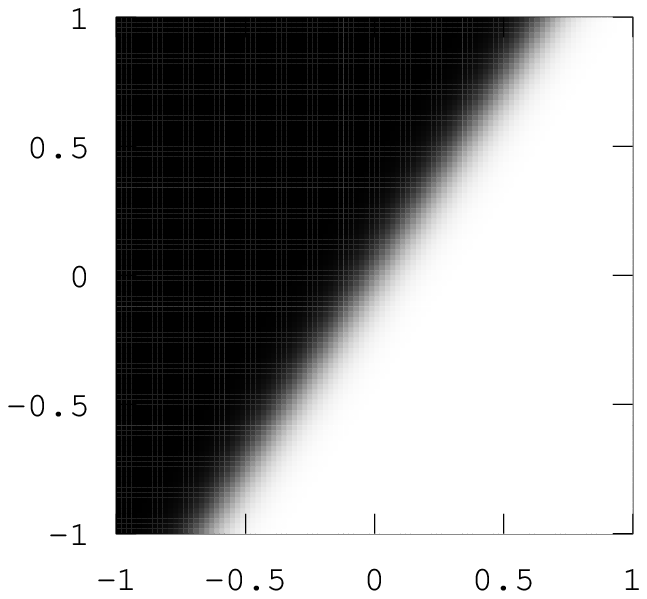}
  \caption{A reference state.\label{fig:refstate0_icond}}
  \end{center}\end{subfigure}
  \begin{subfigure}{0.245\textwidth}\begin{center}
  \includegraphics[width = \textwidth]{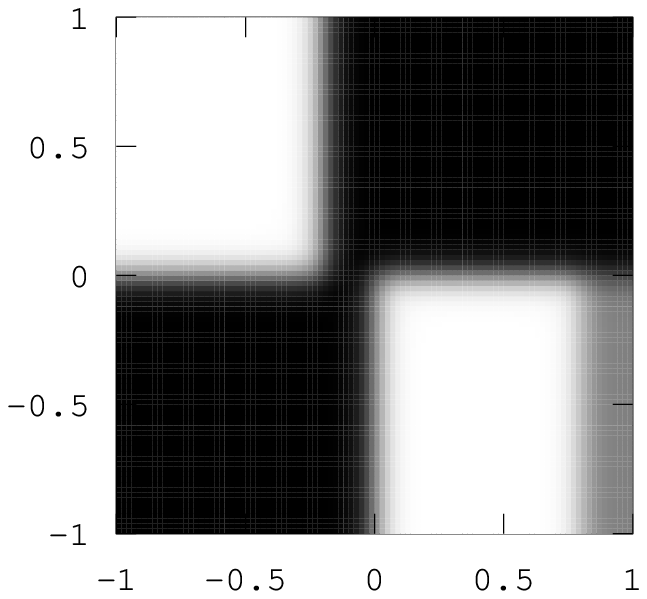}
  \caption{Init. cond., 1st variant\label{fig:icond0_icond}}
  \end{center}\end{subfigure}
  \begin{subfigure}{0.245\textwidth}\begin{center}
  \includegraphics[width = \textwidth]{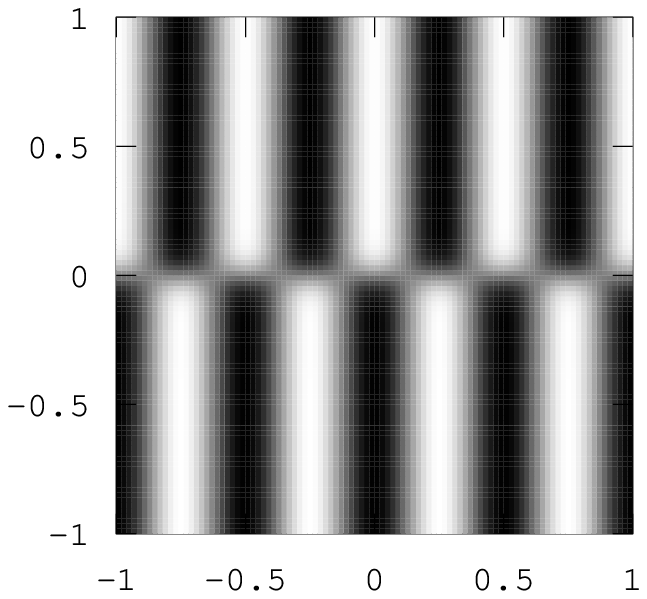}
  \caption{Init. cond., 2nd variant\label{fig:icond1_icond}}
  \end{center}\end{subfigure}
  \caption{A part of data employed for simulations in Section \ref{sec:CP.numexamples.results}.}\label{fig:sub:iconddata}
  \end{center}
\end{figure}

Figures \ref{fig:sub:devconfdata} and \ref{fig:sub:iconddata} present data which shall be utilized in the experiments below. The data employed in particular experiments will be specified in their description by reference to these figures.

Moreover, assume that $y^*\in \Ha{1}$ and that $y$ is the main component of numerical solution of (\ref{eqnmain1num}) - (\ref{eqnmain2num}), obtained with the methods described in Section \ref{sec:CP.numexamples.assumptions.numerics}. For the time discretization points $t=m\tau$, $m=0,\ldots,M$ we denote by $E_{y}(t)$ the $L^2$ error between $y$ and $y^*$:
\begin{displaymath}
E_{y}(t) = \xnorm{y(t) - y^*}{\Lp{2}{\Omega}}
\end{displaymath}
and by $E_{y}^{grad}(t)$ the gradient error between $y$ and $y^*$, or more precisely:
\begin{displaymath}
E_{y}^{grad}(t) = \xnorm{\nabla\left(y(t) - y^*\right)}{\Lp{2}{\Omega}}
\end{displaymath}

The below described simulations have been performed with use of the GNU Octave environment.

\subsubsection{Experiment \arabic{experiment.CP} --- thermostats model and an unstable equilibrium}\label{sec:CP.numexamples.results:exp_unstable}
\addtocounter{experiment.CP}{1}
The present experiment is intended to illustrate properties of the \tcb{thermostat control system} in a situation where the reference state is unstable. We want to emphasize the benefits, which can be observed in this particular experiment, of using the \tcb{thermostats control system} under consideration.

The following data were exploited for the present experiment:
\begin{displaymath}\begin{array}{llll}
T = 24 		&\quad C_g = 16/\pi		&\quad L_w = -10	&\quad\kappa_{j0} = 0\ \forall_{j=1,\ldots,J}\\
D = 0.01 	&\quad C_{switch} = 0.2		&\quad H_w = 10		&\\
\end{array}\end{displaymath}
together with the numerical scheme specification given by:
\begin{displaymath}
N = 100,\qquad M = 2400,\qquad N_{Picard} = 3
\end{displaymath}
We considered the initial condition $y_0$ as on Figure \ref{fig:icond0_icond} and the reference state $y^*\equiv 0$. Note that the $y^*$ taken into account indeed is an unstable state for the assumed reactive term $f$.

We have performed three simulations, basing on various configurations of the control and measurement devices. The cases of $J=16,36,64$, with the devices varying in size and tightly covering the domain have been considered. The utilized devices configurations are presented on the Figures \ref{fig:devices16_a}, \ref{fig:devices36_a} and \ref{fig:devices64_a}. 

In each of the three simulations, some stabilization of the process behavior occurred after an initial period of relatively strong oscillations. It could be observed that after this initial period there emerged certain relatively stable patterns which didn't underwent further big changes. 
However still, some slow evolution of the numerical solution could be observed in longer time horizon for the cases of $16$ ans $36$ devices. 
Nevertheless, the process states eventually achieved for time $t=T$ didn't exhibit any rapid changes so it may be considered as states that are close to certain equilibrium states of the considered model. However, the latter hypothesis require further work for better verification.

Now, let us comment on the quality of \tcb{thermostats control systems} associated with the addressed devices configurations. For many users the result on Figure \ref{fig:20121025_unstable_J16_T4} cannot be considered to be precise for the problem of leading the state of the process to the state $y^*\equiv 0$ in time $T$. Nevertheless, the situation was changing as we were increasing the amount of the devices, keeping their uniform distribution through the domain. Comparing the Figures \ref{fig:20121025_unstable_J16_T4}, \ref{fig:20121025_unstable_J36_T4} and \ref{fig:20121025_unstable_J64_T4} suggests that the greater the number of the control and measurement devices is, the more precise response of the control devices can be expected. This stays consistent with the natural intuition. 

\begin{figure}[htb]
  \begin{center}
  \begin{subfigure}{0.245\textwidth}\begin{center}
  \includegraphics[width = \textwidth]{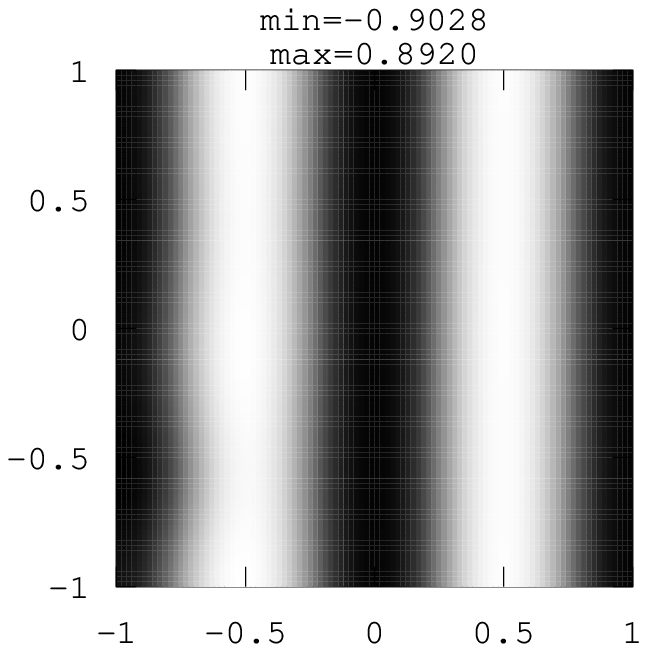}
  \caption{16 dev., $t=T$\label{fig:20121025_unstable_J16_T4}}
  \end{center}\end{subfigure}
  \begin{subfigure}{0.245\textwidth}\begin{center}
  \includegraphics[width = \textwidth]{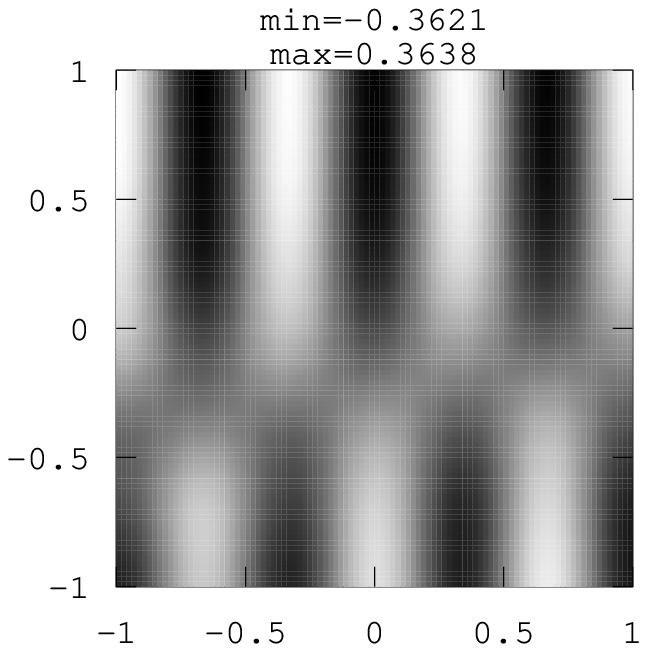}
  \caption{36 dev., $t=T$\label{fig:20121025_unstable_J36_T4}}
  \end{center}\end{subfigure}
  \begin{subfigure}{0.245\textwidth}\begin{center}
  \includegraphics[width = \textwidth]{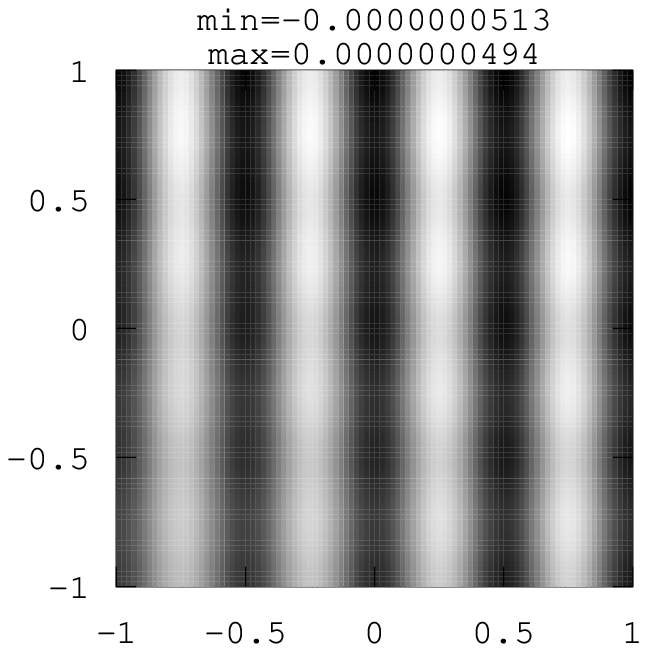}
  \caption{64 dev., $t=T$\label{fig:20121025_unstable_J64_T4}}
  \end{center}\end{subfigure}
  \caption{Solution plots in time $t=T$ for the devices configurations considered in Section \ref{sec:CP.numexamples.results:exp_unstable}. 
  Fig. \ref{fig:20121025_unstable_J16_T4} corresponds to the dev.~conf. in Fig. \ref{fig:devices16_a}; Fig. \ref{fig:20121025_unstable_J36_T4} --- to Fig. \ref{fig:devices36_a}; Fig. \ref{fig:20121025_unstable_J64_T4} --- to Fig. \ref{fig:devices64_a}. 
  The \texttt{min} and \texttt{max} values equal the extremal values of the corresponding solutions in time $t=T$.
  }\label{fig:sub:unstableresults}
  \end{center}
\end{figure}

\begin{figure}[htb]
\begin{center}
\begin{subfigure}{0.325\textwidth}\begin{center}
\includegraphics[width = \textwidth]{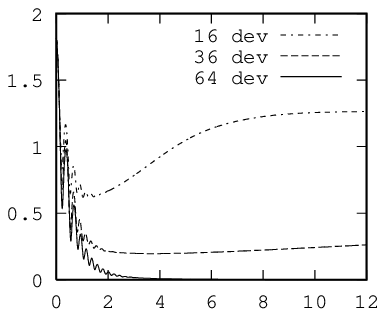}
\caption{$E_y(t)$ values (vert. axis) in time\\\rule{0pt}{1ex}\label{fig:unstable_error_square}}
\end{center}\end{subfigure}
\begin{subfigure}{0.325\textwidth}\begin{center}
\includegraphics[width = \textwidth]{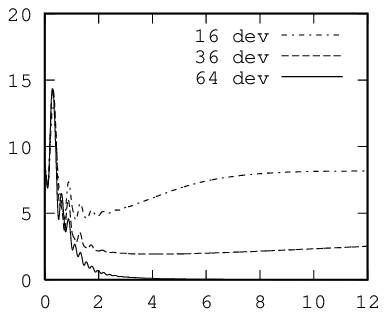}
\caption{$E^{grad}_y(t)$ values (vert. axis) in time\label{fig:unstable_error_gradient}}
\end{center}\end{subfigure}
\caption{$E_y(t)$ and $E^{grad}_y(t)$ for time points $t=m\tau$, $m=0,\ldots,M/2$ for simulations corresponding to the devices configurations considered in Section \ref{sec:CP.numexamples.results:exp_unstable}. For the sake of readability, the time horizon of the error plots is limited to $[0,12]$. After time $t=12$ the error values still evolves, however slowly, without rapid changes.}\label{fig:sub:unstableerror}
\end{center}
\end{figure}

The drastic difference between the efficiency of the \tcb{thermostats control system} for $16$ devices and the efficiency for the cases of $36$ and $64$ devices is well visible on the error plots in Figures \ref{fig:unstable_error_square} and \ref{fig:unstable_error_square}. The Reader may also compare the obtained error values in time $t=T$ in the below table (presented with rounding to $4$ significant digits, or more if necessary):
\begin{center}
\begin{tabular}{|l|r|r|r|}\hline
$y$ part for:		& $16$ dev. & $36$ dev. & $64$ dev. \\
\hline
$E_{y}(T)$		&1.3006 &0.3568 &5.5550e-08\\
$E_{y}^{grad}(T)$	&8.2791 &3.4143 &6.9999e-07\\
\hline\end{tabular}
\end{center}

Nevertheless, in the situation of the present experiment, the main question concerning the efficiency of the utilized control system can be reduced to the question on the amount of the devices which would be sufficient to achieve demanded precision. This is much simpler adjustment procedure than what can be often expected in case of the open-loop control systems. Suppose that we consider an open-loop system in which the user is responsible for the choice of right amount, size and locations of the control devices as well as the signal functions $\kappa_j$ given on $[0,T]$. In other words, equations (\ref{eqnmain2num}) are not taken into account. Such an open-loop control system is more difficult to handle that our closed-loop control model in (\ref{eqnmain1num}) - (\ref{eqnmain2num}), because the user has to control more variables. Necessary is the choice of the devices together with the signal functions in the introduced open-loop case, versus the choice of the devices only in the case of our closed-loop system. Moreover, in the open-loop situation a proper choice of the signal functions $\kappa_j$ is hard to be done by intuition. Probably proper signal would be searched by some optimization procedure, what additionally increases the complexity of efforts necessary to deal with the open-loop case.

\subsubsection{Experiment \arabic{experiment.CP} --- thermostats model and various initial conditions}\label{sec:CP.numexamples.results:exp_icond}
\addtocounter{experiment.CP}{1}
Below, we present numerical results which illustrate behavior occurring in the investigated model with control by thermostats when perturbations of the initial condition are induced.

In the present experiment, the following data were used :
\begin{displaymath}\begin{array}{llll}
T = 4 		&\quad r_\sigma = 1/8	&\quad L_w = -10	&\quad C_{switch} = 0.2\\
D = 0.02 	&\quad C_g = 16/\pi  	&\quad H_w = 10		&\quad\kappa_{j0} = 0\ \forall_{j=1,\ldots,J}\\
\end{array}\end{displaymath}
together with the numerical scheme specification given by:
\begin{displaymath}
N = 100,\qquad M = 400,\qquad N_{Picard} = 3
\end{displaymath}
The configuration of the control and measurement devices was assumed to be as the devices configuration with $J=64$ utilized in the experiment from the Section \ref{sec:CP.numexamples.results:exp_unstable}, i.e. as on Figure \ref{fig:devices64_a}. The reference state was as in Figure \ref{fig:refstate0_icond}.

Two simulations has been performed, with two variants of the $y_0$ component of the initial condition. The first of them was as in Figure \ref{fig:icond0_icond}, the second initial condition was as in Figure \ref{fig:icond1_icond}.

\begin{figure}[htb]
\begin{center}
\begin{subfigure}{0.245\textwidth}\begin{center}
\includegraphics[width = \textwidth]{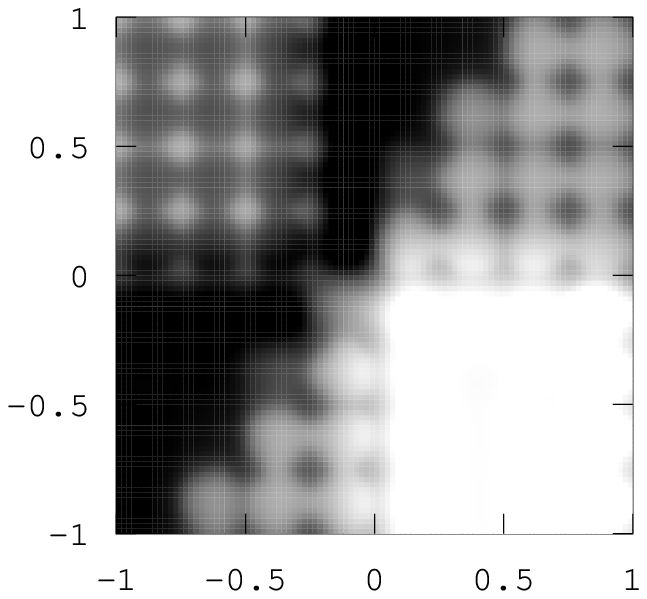}
\caption{1st variant, $t=0.25$\label{fig03icond}}
\end{center}\end{subfigure}
\begin{subfigure}{0.245\textwidth}\begin{center}
\includegraphics[width = \textwidth]{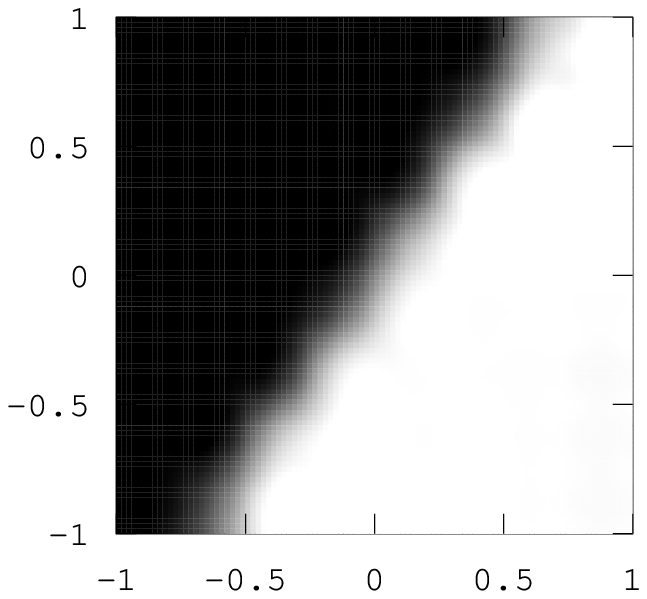}
\caption{1st variant, $t=1$\label{fig04icond}}
\end{center}\end{subfigure}
\begin{subfigure}{0.245\textwidth}\begin{center}
\includegraphics[width = \textwidth]{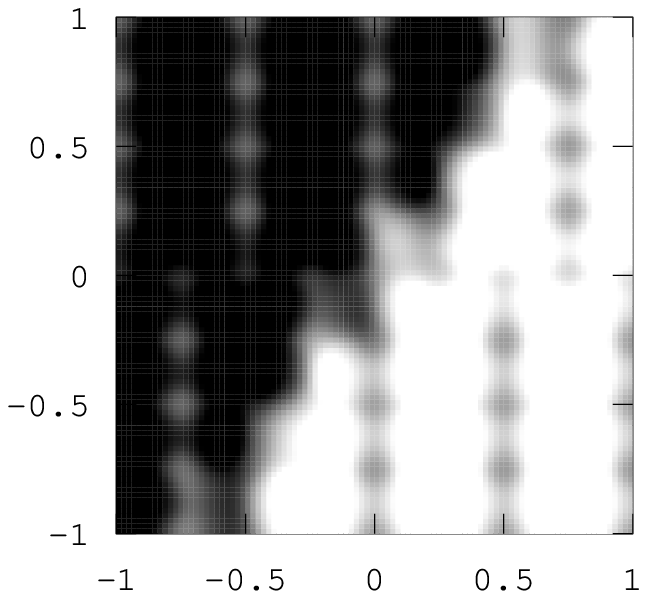}
\caption{2nd variant, $t=0.25$\label{fig05icond}}
\end{center}\end{subfigure}
\begin{subfigure}{0.245\textwidth}\begin{center}
\includegraphics[width = \textwidth]{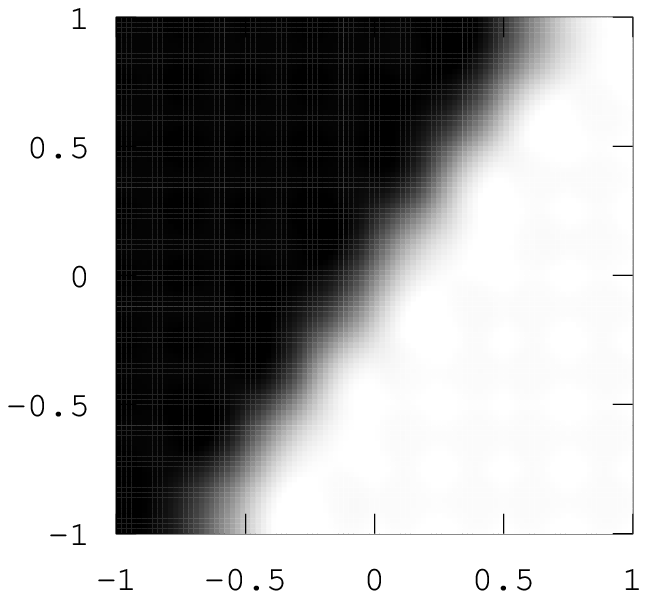}
\caption{2nd variant, $t=1$\label{fig06icond}}
\end{center}\end{subfigure}
\caption{Numerical solutions at time $t=0.25$ and $t=1$, for two initial condition variants considered in Section \ref{sec:CP.numexamples.results:exp_icond}. 
Fig. \ref{fig03icond}, \ref{fig04icond} correspond to the i.~cond. in Fig. \ref{fig:icond0_icond}; Fig. \ref{fig05icond}, \ref{fig06icond} --- to Fig. \ref{fig:icond1_icond}.
}\label{fig:sub:icondresults}
\end{center}
\end{figure}

For the both simulations, some stabilization occurred after initial period of oscillations. Stable patterns emerging after this initial period seemed to match the reference state at some rate of accuracy, at least visually. Moreover, the numerical solution generated in both simulations occurred to achieve a high level of likeness in a short time. This is visible on Figures \ref{fig03icond} - \ref{fig06icond} --- perhaps, if we had swapped the Figures \ref{fig04icond} and \ref{fig06icond}, it would be hard to notice for the Reader. So it can be a reasonable hypothesis that the rate of approximation of the reference state in the final time of the experiment is similar for two subject cases.
\begin{figure}[htb]
\begin{center}
\begin{subfigure}{0.325\textwidth}\begin{center}
\includegraphics[width = \textwidth]{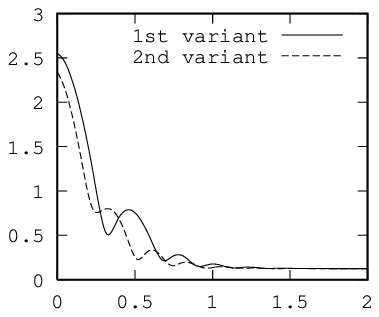}
\caption{$E_y(t)$ values (vert. axis) in time\\\rule{0pt}{1ex}\label{fig07icond}}
\end{center}\end{subfigure}
\begin{subfigure}{0.325\textwidth}\begin{center}
\includegraphics[width = \textwidth]{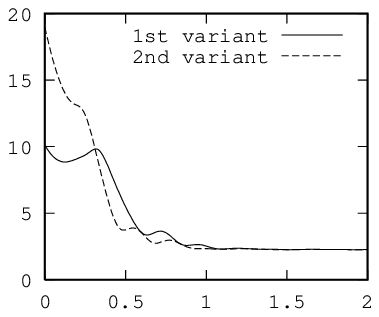}
\caption{$E^{grad}_y(t)$ values (vert. axis) in time\label{fig08icond}}
\end{center}\end{subfigure}
\caption{$E_y(t)$ and $E^{grad}_y(t)$ for time points $t=m\tau$, $m=0,\ldots,M/2$, for simulations corresponding to the two initial condition variants considered in Section \ref{sec:CP.numexamples.results:exp_icond}. The time interval for the plots is limited to $[0,2]$ for the sake of readability. No significant fluctuations of the error values were observed after time $t=2$. }\label{fig:sub:iconderror}
\end{center}
\end{figure}

The error plots in Figures \ref{fig07icond} and \ref{fig08icond} confirm that the components $y$ of the both numerical solutions tend to the same neighborhood of the reference state. 
 Moreover, the below table presents the values of error in the final time of the experiment for the respective initial conditions (with rounding to $8$ significant digits):
\begin{center}
\begin{tabular}{|l|r|r|r|}\hline
$y$ part for:	& 1st variant & 2nd variant & ratio \\
\hline
$E_{y}(T)$		&0.12569814&0.12569916&1.00000812\\
$E_{y}^{grad}(T)$	&2.26541586&2.26541453&0.99999941\\
\hline\end{tabular}
\end{center}
The ratio of the error in time $t=T$ is close to $1$. Thus, depending on particular requirements of the user of the control system, the achieved accuracy in the final time of the experiment can be considered to be similar for the two initial condition cases. 

Moreover, the error plots in the Figures \ref{fig07icond} and \ref{fig08icond} suggest that the initial error was leveled within a similar time, approximately equal $t\approx 1$, in both cases. Note however, that the latter can be guessed to be true due to the comparable rank of values of the considered initial conditions. It is reasonable to expect that if we had considered two initial conditions where one of them was defined as ten thousand times the other then the time of leveling the initial error would differ.

To conclude this experiment results, assume that we want to measure the efficiency of the control system as deviation between the solution and the reference state in the final time of the experiment. Then, a hypothesis that the \tcb{thermostats control system} under consideration have the very nice property of preserving the efficiency under perturbations of the initial condition seems reasonable. 

\subsubsection{Experiment \arabic{experiment.CP} --- various amounts of thermostats}\label{sec:CP.numexamples.results:exp_linefb}
\addtocounter{experiment.CP}{1}
This experiment is devoted to compare behavior of the \tcb{thermostats control system} for two different configurations of the control and measurement devices, where the size of the particular devices equals in both cases but their amount differs. The results can serve as a motivation for the research in the area of optimal location of the control and measurement devices.

The following data was exploited for the present experiment:
\begin{displaymath}\begin{array}{llll}
T = 4 		&\quad r_\sigma = 1/8	&\quad L_w = -10	&\quad C_{switch} = 0.2\\
D = 0.02 	&\quad C_g = 16/\pi  	&\quad H_w = 10		&\quad\kappa_{j0} = 0\ \forall_{j=1,\ldots,J}\\
\end{array}\end{displaymath}
together with the numerical scheme specification given by:
\begin{displaymath}
N = 100,\qquad M = 400,\qquad N_{Picard} = 3
\end{displaymath}
The initial condition chosen for the present experiment was as in Figure \ref{fig:icond0_icond}) and the reference state was as in Figure \ref{fig:refstate0_icond}.

The two considered configurations of the control and measurement devices, one with $J=64$ and the other with $J=20$, are presented in Figures \ref{fig:devices64_a} and \ref{fig:devices20_a}. 

\begin{figure}[htb]
\begin{center}
\begin{subfigure}{0.245\textwidth}\begin{center}
\includegraphics[width = \textwidth]{img_J64_icond0_m101_400x400gray_meanflat.eps}
\caption{64 dev., $t=1$.\label{fig03}}
\end{center}\end{subfigure}
\begin{subfigure}{0.245\textwidth}\begin{center}
\includegraphics[width = \textwidth]{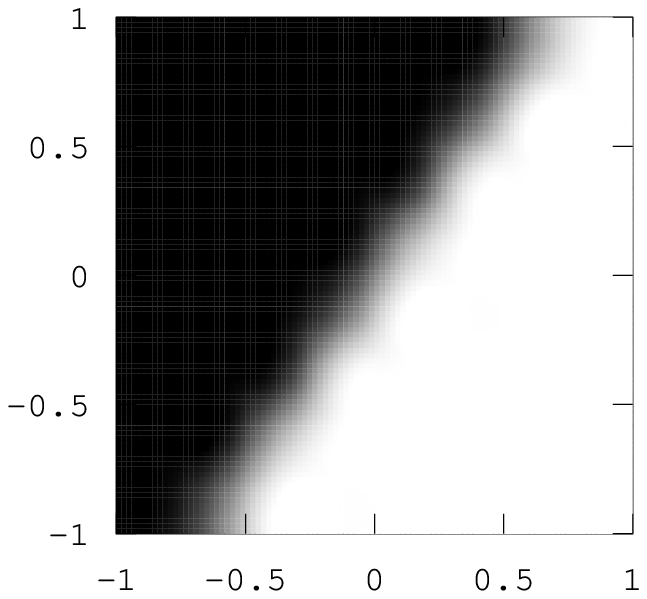}
\caption{64 dev., $t=2$.\label{fig04}}
\end{center}\end{subfigure}
\begin{subfigure}{0.245\textwidth}\begin{center}
\includegraphics[width = \textwidth]{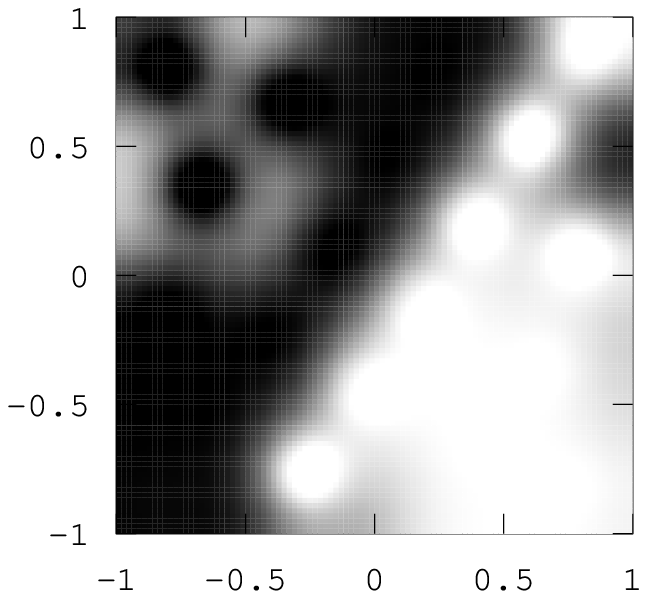}
\caption{20 dev., $t=1$.\label{fig05}}
\end{center}\end{subfigure}
\begin{subfigure}{0.245\textwidth}\begin{center}
\includegraphics[width = \textwidth]{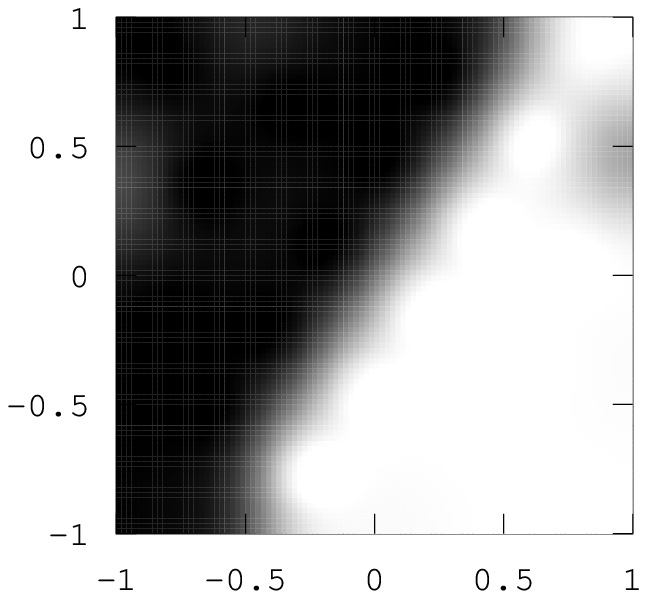}
\caption{20 dev., $t=2$.\label{fig06}}
\end{center}\end{subfigure}
\caption{Numerical solutions at time $t=1$ and $t=2$ for the devices configurations considered in Section \ref{sec:CP.numexamples.results:exp_linefb}. 
Fig. \ref{fig03}, \ref{fig04} correspond to the dev.~conf. in Fig. \ref{fig:devices64_a}; Fig. \ref{fig05}, \ref{fig06} --- to Fig. \ref{fig:devices20_a}.
}\label{fig:sub:linefbresults}
\end{center}
\end{figure}

The numerical approximation of the process occurred to stabilize quickly near the reference state for the case of $64$ devices, see Figures \ref{fig03} and \ref{fig04}. For the case of $20$ devices, as we see on the Figures  \ref{fig05} and \ref{fig06}, the process also seems to tend to some neighborhood of the reference state. However for $20$ devices, the evolution toward the reference state is slower than in case of the simulation with $64$ devices. 

\begin{figure}[htb]
\begin{center}
\begin{subfigure}{0.325\textwidth}\begin{center}
\includegraphics[width = \textwidth]{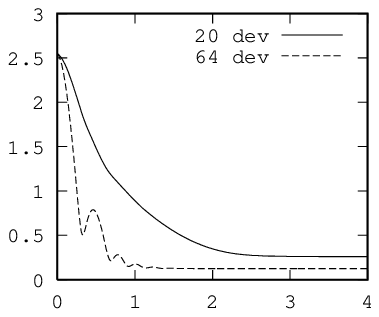}
\caption{$E_y(t)$ values (vert. axis) in time\\\rule{0pt}{1ex}\label{fig07}}
\end{center}\end{subfigure}
\begin{subfigure}{0.325\textwidth}\begin{center}
\includegraphics[width = \textwidth]{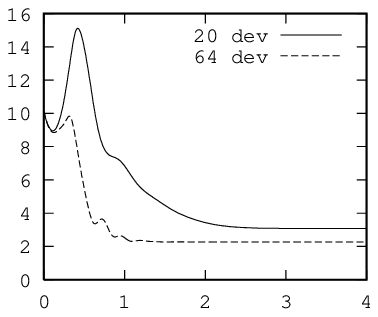}
\caption{$E^{grad}_y(t)$ values (vert. axis) in time\label{fig08}}
\end{center}\end{subfigure}
\caption{$E_y(t)$ and $E^{grad}_y(t)$ for time points  $t=m\tau$, $m=0,\ldots,M$, for simulations corresponding to the devices configurations considered in Section \ref{sec:CP.numexamples.results:exp_linefb}.}\label{fig:sub:linefberror}
\end{center}
\end{figure}

In Figures \ref{fig07} and \ref{fig08}, the error line concerning the behavior of the process in the case of $20$ control and measurement devices is always above the error line concerning $64$ control and measurement devices, in both plots. In other words --- for our situation, the error graphs stay consistent with the hypothesis that the \tcb{thermostats control system} with $20$ devices works slower than the one with $64$ devices. Moreover, the error plots suggest that the efficiency of the \tcb{thermostats control system} understood as error in time $t=T$ also differ for these two devices configurations. 
The latter is also confirmed by the below table containing the error values in time $t=T$ (with rounding to $4$ significant digits):
\begin{center}
\begin{tabular}{|l|r|r|r|}\hline
$y$ part for:	& $64$ dev. & $20$ dev. & ratio \\
\hline
$E_{y}(T)$	&0.2609&0.1257&0.4817\\
$E_{y}^{grad}(T)$	&3.0757&2.2654&0.7366\\
\hline\end{tabular}
\end{center}

To sum up the present experiment, the $20$ devices system seems to be slower and not that accurate as the $64$ devices system. 
But these results lead to further questions. The configuration of the $20$ control and measurement devices presented on Figure \ref{fig:devices20_a} has been chosen for our experiments by intuition. Hence it is natural to ask whether these devices could be localized in the domain $\Omega$ better. Or, whether we could remove more control devices and still obtain a result which would be called satisfactory with respect to some criterion. Here the realm of optimization begins.

\subsubsection{Remarks on large time behavior}\label{sec:CP.numexamples.results:asymptotics}
In the above described experiments, an observation about stabilization of the process in time was made. This allows to pose hypotheses on asymptotic properties for model (\ref{eqnmain1}) - (\ref{eqnmain3}). 

However, still it is not straightforward what should be the precise form of the hypotheses in question. The numerical prototypes in Sections \ref{sec:CP.numexamples.results:exp_unstable}, \ref{sec:CP.numexamples.results:exp_icond} and \ref{sec:CP.numexamples.results:exp_linefb} show that the behavior of the model varies depending on the configuration of its particular parameters. Intuition suggests that in the situations taken into account in the simulations in Sections \ref{sec:CP.numexamples.results:exp_icond} and \ref{sec:CP.numexamples.results:exp_linefb} the process evolves toward certain state which is relatively close to the reference state. 
For these configurations of the model existence of a one-point, or a very small attractor can be expected. 

Nevertheless, the situation in the simulations concerning the unstable equilibrium, what was the case in Section \ref{sec:CP.numexamples.results:exp_unstable}, is different. If the process states in final time presented on the Figures \ref{fig:20121025_unstable_J16_T4}, \ref{fig:20121025_unstable_J36_T4} and \ref{fig:20121025_unstable_J64_T4} are close to certain equilibrium of the process then, by symmetry, the transposed states are close to an equilibrium of the system as well. The transposed state should be the output for a simulation with transposed initial condition. By a transposed state we mean a state with swapped role of axis of the coordinate system in $\R^2$. 
In particular, in the case of the configuration of the model associated with Figure \ref{fig:devices16_a} the hypothetic attractor cannot be expected to be small in the sense of diameter, if exists. The reason for this is that in the subject case the state in the final time presented on Figure \ref{fig:20121025_unstable_J16_T4} is quite distant from its transposed state. The attractor, if exists, should contain states which are close to both the original and transposed state. 

Moreover, differences between particular simulations occur also in terms of speed of the evolution of the process toward some time invariant state. For simulations described in Sections \ref{sec:CP.numexamples.results:exp_icond}, time interval $[0,4]$ was enough for the process to achieve some state that seemed time invariant. This was also reflected on the error plots on Figures \ref{fig07icond}, \ref{fig08icond}, \ref{fig07}, \ref{fig08}. In contrary, for unstable reference state considered in experiment in Section \ref{sec:CP.numexamples.results:exp_unstable} and the cases of $J=16$ and $J=36$ control and measurement devices (Figures \ref{fig:devices16_a} and \ref{fig:devices36_a}) the evolution of the process toward the state obtained in the final time of the experiment was very slow. This is the main reason for which we have chosen the time interval for this experiment equal to $[0,24]$, what is six times longer than the time intervals in other experiments. At time $t=4$ the numerical approximations of the process in Section \ref{sec:CP.numexamples.results:exp_unstable} for $J=16$ and $J=36$ still evolved. This is visible in the error plots in Figures \ref{fig:unstable_error_square} and \ref{fig:unstable_error_gradient}. Nevertheless, the above hypotheses concerning the speed of evolution of the process base on visual inspection of the results and require further verification.

\section{Concluding remarks}\setcounter{equation}{0}
Motivated by the results in Section \ref{sec:CP.numexamples.results:exp_linefb} and according to what we stated in the introduction to this paper, we are interested in further research in the field of optimal location of control and measurement devices in the considered model. This research, which is partially completed now, takes into its scope both theoretical analysis including deriving optimality conditions as well as the numerical optimization. We intend to contain these results in our next work (in preparation). 

The problem of optimal location is not the only possible field for continuation of our work. As we already remarked in Section \ref{sec:CP.numexamples.results:asymptotics}, the investigated model with control by thermostats exhibits some asymptotic behavior and this behavior can vary as we change the model parameters. A precise description of large time properties of the model could be an additional subject for further research. The existence of a global attractor and periodic solutions for a model with control by thermostats was investigated in \cite{gj1}. However, there addressed variant of the \tcb{thermostats control system} was different than ours. This makes the possibility of transmitting there presented results on global attractors and periodic solutions to our case not obvious.

It is worth of a short remark that the questions on the asymptotic behavior 
can play a role also from the point of view of the problem of optimal location of control and measurement devices. Any optimal control problem requires an optimality criterion, describing quality of a given control. If asymptotic stabilization of the process occurs, then a valuable quality criterion of a given configuration of control and measurement devices can be the deviation between the actual process state and the reference state in the final time of the experiment. 
Still, if the global attractor 
exists, the independence of the latter quality criterion on the initial condition cannot be expected unless the attractor contains only one point or at least has very small diameter. This kind of independence is a desired property in situations when a problem of choosing locations for control and measurement devices arises. 
For example, consider a temperature stabilization problem consisting in choosing optimal locations of control and measurement devices in given environment. The aim is to bring the temperature distribution in the environment possibly close to some prescribed temperature distribution, without \textit{a priori} knowledge on the initial temperature distribution. It is natural demand to locate the control and measurement devices in a way such that the final temperature distribution doesn't contain unpredictable bias. In the case of nontrivial asymptotic behavior, bias in the final temperature distribution can occur if perturbations in the initial temperature distribution occur.

\section{Acknowledgements}
\fmacknowledgement

{}

\end{document}